\documentclass[a4paper, 11pt]{amsart}
\usepackage{amsthm}
\usepackage[]{amsmath}
\usepackage{amssymb}
\usepackage{enumerate}
\usepackage{tabularx}
\usepackage[]{color}
\usepackage[left=3cm, right=3cm]{geometry}
\usepackage[colorlinks]{hyperref}
\usepackage{tikz}
\usepackage{multirow}
\usepackage{subcaption}
\usepackage{algorithm}
\usepackage{algorithmic}
\usepackage{multirow}

\newcommand{\bm}[1]{\boldsymbol{#1}}
\newcommand{\bmr}[1]{\bm{\mr{#1}}}
\newcommand{\lj}{[ \hspace{-2pt} [}
\newcommand{\rj}{] \hspace{-2pt} ]}
\newcommand{\mb}[1]{\mathbb{#1}}
\newcommand{\mc}[1]{\mathcal{#1}}
\newcommand{\mr}[1]{\mathrm{#1}}
\newcommand{\jump}[1]{\lj #1 \rj}

\newcommand{\wt}[1]{ \widetilde{ #1}}
\newcommand{\tr}[1]{\ifmmode \mathrm{tr}\left( #1 \right) \else 
\text{tr} \left( #1 \right) \fi }

\newcommand\enorm[1]{|\!|\!| #1 |\!|\!|}

\renewcommand{\d}[1]{\mathrm d \boldsymbol{#1}}
\newcommand\pnorm[1]{\| #1 \|_{\bm{\mathrm{p}}}}
\newcommand\unorm[1]{\| #1 \|_{\bm{\mathrm{u}}}}

\newcommand\comment[1]{}

\def\MTh{\mc{T}_h}
\def\MEh{\mc{E}_h}
\def\MFh{\mc{F}_h}
\def\un{\bm{\mr{n}}}
\def\curl{\ifmmode \mathrm{curl} \else \text{curl}\fi}
\def\Ned{\ifmmode \text{N\'ed\'elec} \else \text{N\'ed\'elec} \fi}

\newtheorem{theorem}{Theorem}
\newtheorem{lemma}{Lemma}

\newtheorem{remark}{Remark}

\title[LS Method for Maxwell]{A Discontinuous Least Squares Finite
  Element Method for Time-Harmonic Maxwell Equations}

\author[R. Li]{Ruo Li} \address{CAPT, LMAM and School of Mathematical
  Sciences, Peking University, Beijing 100871, P.R. China}
\email{rli@math.pku.edu.cn}

\author[Q.-C. Liu]{Qicheng Liu} \address{School of Mathematical
  Sciences, Peking University, Beijing 100871, P.R. China}
\email{qcliu@pku.edu.cn}

\author[F.-Y. Yang]{Fanyi Yang} \address{School of Mathematical
  Sciences, Peking University, Beijing 100871, P.R. China}
\email{yangfanyi@pku.edu.cn}

\begin{document}

\begin{abstract}
We propose and analyze a discontinuous least squares finite element
method for solving the indefinite time-harmonic Maxwell equations.
The scheme is based on the $L^2$ norm least squares functional with the
weak imposition of the continuity across the interior faces. We
minimize the functional over the piecewise polynomial spaces to seek
numerical solutions. The method is shown to be stable without any
constraint on the mesh size.  We prove the convergence orders under
both the energy norm and the $L^2$ norm.  Numerical results in two and
three dimensions are presented to verify the error estimates.

\noindent \textbf{keywords}: Time-harmonic Maxwell's equations, Least
squares method, Discontinuous elements.
\end{abstract}


\maketitle

\section{Introduction}
\label{sec_introduction}
The time-harmonic Maxwell equations are often encountered in many
engineering applications such as antenna design, microwaves, and
satellites \cite{Hiptmair2002finite, Nguyen2011hybridizable}. These
applications require us to carry out numerical studies of the
time-harmonic Maxwell equations. The Maxwell's operator is strongly
indefinite especially for the case of the high wave number, which
brings many difficulties in the numerical simulation and the error
analysis \cite{Lu2017absolutely}. Despite these difficulties,
there are plenty of studies on numerical methods for this problem,
such as finite difference methods, spectral methods, and finite
element methods.

There are a variety of finite element methods for solving the
time-harmonic Maxwell equations and related problems. A common
choice is to employ the $H(\curl)$-conforming elements, known as edge
elements. We refer to \cite{Nedelec1986new, Nedelec1980mixed,
Monk2003finite, Hiptmair2002finite, Chen2007adaptive,
Jiang2013adaptive} for more details of these conforming methods. We
note that the implementation of edge elements is still challenging
especially in the higher-order case \cite{Bramble2005least}.

Using discontinuous functions to approximate the solution, the
discontinuous Galerkin (DG) methods have been applied in the numerical
simulation of the Maxwell equations. The DG methods can offer great
flexibility in the mesh structure which allows the elements of
different shapes and can easily handle the irregular non-conforming
meshes. Additionally, the implementation of discontinuous elements is
very straightforward and there is no need to use the curl-conforming
elemental mappings \cite{Lu2017absolutely}. We refer to
\cite{Lu2017absolutely, Feng2014absolutely, Houston2005interior,
Sarmany2010optimal, Nguyen2011hybridizable, Brenner2007locally,
Perugia2002stabilized, Perugia2003local, Chen2019analysis} and the
references therein for those DG methods.

The least squares finite element method (LSFEM) is a general technique
in numerical PDEs which is based on the minimization of the $L^2$ norm
of the residual over an approximation space, and we refer to the paper
\cite{Bochev1998review} for an overview. The LSFEM has also been
applied to solve the Maxwell system and we refer to
\cite{Bramble2005least, Bramble2005approximation, Jagalur2013galerkin,
Hu2019variant, Hu2014plane} for such least squares methods. One of
the advantages of LSFEM is that the resulting linear system is always
symmetric and positive definite. Recently, the LSFEM has been extended
for the numerical approximation to some classical PDEs using
discontinuous elements, and this method is referred to as the
discontinuous LSFEM. We refer the readers to \cite{Bensow2005div,
Bensow2005discontinuous, Bochev2013nonconforming, Bochev2012locally,
li2019least, li2019sequential} for more details.

Noticing that the Maxwell's operator is indefinite while the LSFEM
always provides a positive definite linear system, we are motivated to
develop a stable numerical method, based on the least squares
functional and discontinuous elements, to solve the time-harmonic
Maxwell equations. For this purpose, we define an $L^2$ norm least
squares functional involving the proper penalty terms which weakly
enforce the tangential continuity across the interior faces as well as
the boundary conditions. The functional is then minimized over the
discontinuous piecewise polynomial spaces to seek the numerical
solutions. The method is easy to be implemented and the resulting
linear system is shown to be symmetric and positive definite. The
method combines the attractive features of DG methods and LSFEMs,
giving us a new discontinuous least squares finite element method for
the time-harmonic Maxwell equations.

We estimate the error of the new method to derive the convergence
rates for all solution variables under both the $L^2$ norms and the
energy norms. It is interesting that the proposed method is shown to
be unconditionally stable without making any assumptions about the
mesh size. We carry out a series of numerical tests in two and three
dimensions to verify the theoretical predictions. Noticing that the
least squares functional can serve as an {\it a posteriori} estimator,
we particularly present a low-regularity example, which is solved by
the $h$-adaptive refinement strategy using the least squares 
functional as an adaptive indicator.

The rest of this paper is organized as follows. In Section
\ref{sec_preliminaries}, we introduce the notations and define the
first-order system to the time-harmonic Maxwell equations. In Section
\ref{sec_method}, we present our least squares method and define the
least squares functional. The error estimates are also proven in this
section. In Section \ref{sec_numericalresults}, we present a series of
numerical examples to illustrate the accuracy of the proposed method.


\section{Preliminaries} 
\label{sec_preliminaries}
Let $\Omega$ be an open, bounded polygonal (polyhedral) domain with
the boundary $\partial \Omega$ in $\mb{R}^d$, $d = 2, 3$. We denote by
$\MTh$ a regular and shape-regular partition over the domain $\Omega$
into triangles (tetrahedrons). We let $\MFh^i$ be the collection of
all $d - 1$ dimensional interior faces with respect to the partition
$\MTh$, and we let $\MFh^b$ be the collection of all $d - 1$
dimensional faces that are on the boundary $\partial \Omega$, and then
we set $\MFh := \MFh^i \cup \MFh^b$. In particular, for the
three-dimensional case we denote by $\MEh^i$ the set of all $d - 2$
dimensional interior edges of all elements in $\MTh$, and by $\MEh^b$
the set of all $d - 2$ dimensional boundary edges, and we let $\MEh :=
\MEh^i \cup \MEh^b$. For the element $K \in \MTh$ and the face $f \in
\MFh$, we let $h_K$ and $h_f$ be their diameters, respectively, and we
denote $h := \max_{K \in \MTh} h_K$ as the mesh size of $\MTh$.  Then,
the shape-regularity of $\MTh$ is in the sense of that there exists a
constant $C > 0$ such that for any element $K \in \MTh$, 
\begin{displaymath}
  \frac{h_K}{\rho_K} \leq C,
\end{displaymath}
where $\rho_K$ denotes the diameter of the largest disk (ball)
inscribed in $K$. 

Next, we introduce the following trace operators that are commonly
used in the DG framework. Let $f \in \MFh^i$ be an interior face
shared by two adjacent elements $K^+$ and $K^-$ with the unit outward
normal vectors $\un^+$ and $\un^-$ on $f$, respectively. For the
piecewise smooth scalar-valued function $v$ and the piecewise smooth
vector-valued function $\bm{q}$, we define the jump operator $\jump{
\cdot}$ on $f$ as 
\begin{displaymath}
  \jump{\un \times v} := \un^+ \times v|_{K^+} +  \un^-\times
  v|_{K^-}, \quad \jump{\un \times \bm{q}} := \un^+ \times
  \bm{q}|_{K^+} + \un^- \times \bm{q}|_{K^-}.
\end{displaymath}
For a boundary face $f \in \MFh^b$, we let $K \in \MTh$ be the element
that contains $f$ and the jump operator $\jump{\cdot}$ on $f$ is
defined as
\begin{displaymath}
  \jump{v} := \un \times v|_K, \quad \jump{\un \times \bm{q}} := \un
  \times \bm{q}|_{K},
\end{displaymath}
where $\un$ is the unit outward normal on $f$.

We note that the capital $C$ with or without subscripts are generic
positive constants, which are possibly different from line to line,
are independent of the mesh size $h$ but may depend on the wave number
$k$ and the domain $\Omega$. Given a bounded domain
$D \subset \Omega$, we will follow the standard notations for the
Sobolev spaces $L^2(D)$, $L^2(D)^d$, $H^r(D)$ and $H^r(D)^d$ with the
regular exponent $r \geq 0$. We also use the standard definitions of
their corresponding inner products, semi-norms and norms. Throughout
the paper, we mainly use the notation for the three-dimensional
case. For the case $d = 2$, it is natural to identity the space
$\mb{R}^2$ with the $(x, y)$ plane in $\mb{R}^3$.  Specifically, in
two dimensions for the vector-valued function $\bm{u} = (u_1, u_2)^T$,
the curl of $\bm{u}$ reads
\begin{displaymath}
  \nabla \times \bm{u} = \frac{\partial u_2 }{\partial x} -
  \frac{\partial u_1}{\partial y}, 
\end{displaymath}
and for the scalar-valued function $q$, we let $\nabla \times q$ be
the formal adjoint, which reads
\begin{displaymath}
  \nabla \times q = \left( \frac{\partial q}{\partial y},
  - \frac{\partial q}{\partial x} \right)^T.
\end{displaymath}
For the problem domain $\Omega$, we define the space 
\begin{displaymath}
  \begin{aligned}
    H(\curl) &:= \begin{aligned}
      &\left\{ v \in L^2(\Omega) \ | \  \nabla \times v \in
      L^2(\Omega)^2 \right\},  \text{ for scalar-valued functions}, \\
      &\left\{ \bm{v} \in L^2(\Omega)^2 \ | \ \nabla \times \bm{v} \in
      L^2(\Omega) \right\},  \text{ for vector-valued functions}, \\
    \end{aligned} \quad d = 2, \\
    H(\curl) & := \left\{ \bm{v} \in L^2(\Omega)^3 \ | \ \nabla \times
    \bm{v} \in L^2(\Omega)^3 \right\}, \quad d = 3. \\
  \end{aligned}
\end{displaymath}
Further, we denote by $H_0(\curl)$ the space of functions in
$H(\curl)$ with vanishing tangential trace, 
\begin{displaymath}
  H_0(\curl) := \left\{ \bm{v} \in H(\curl)\ | \  \un \times
  \bm{v} = 0, \text{ on } \partial \Omega \right\}.
\end{displaymath}
For the partition $\MTh$, we will use the notations and the
definitions for the broken Sobolev space $L^2(\MTh)$, $L^2(\MTh)^d$,
$H^r(\MTh)$ and $H^r(\MTh)^d$ with the exponent $r \geq 0$ and their
associated inner products and norms \cite{arnold2002unified}.

The problem under our consideration is to find a numerical
approximation to the time-harmonic Maxwell equations in a lossless
medium with an inhomogeneous boundary condition, which seeks the
electric field $\bm{u}(\bm{x})$ such that
\begin{equation}
  \begin{aligned}
    \nabla \times \mu_r^{-1} \nabla \times \bm{u} - k^2 \varepsilon_r
    \bm{u} &= \bm{f}, \quad \text{in } \Omega, \\
    \un \times \bm{u} &= \bm{g}, \quad \text{on } \partial \Omega. \\
    \label{eq_Maxwell}
  \end{aligned}
\end{equation}
Here $\bm{f} \in L^2(\Omega)^d$ is an external source file and $k > 0$
is the wave number with the assumption $k$ is not an eigenvalue of the
Maxwell system \cite{Li2017maxwell, Monk2003finite,
Hiptmair2011stability}. $\mu_r$ and $\varepsilon_r$ are the relative
magnetic permeability and the relative electric permittivity of the
medium. For simplicity, we set $\mu_r = 1$ and $\varepsilon_r = 1$. 

Below let us propose a least squares method for the equations
\eqref{eq_Maxwell} based on the discontinuous approximation. We first
introduce an auxiliary variable $\bm{p} = \frac{1}{k} \nabla \times
\bm{u}$ to rewrite the Maxwell equations \ref{eq_Maxwell} into an
equivalent first-order system:
\begin{equation}
  \begin{aligned}
    \nabla \times \bm{p} - k \bm{u} = \wt{\bm{f}}, &\quad \text{in }
    \Omega, \\
    \nabla \times \bm{u} - k \bm{p} = \bm{0}, &\quad \text{in }
    \Omega, \\
    \un \times \bm{u} = \bm{g} & \quad \text{on } \partial \Omega,
  \end{aligned}
  \label{eq_firstMaxwell}
\end{equation}
where $\wt{\bm{f}} = \frac{1}{k} \bm{f}$. We note that to rewrite the
problem into a first-order system is the fundamental idea in the
modern least squares finite element method \cite{Bochev1998review},
and our discontinuous least squares method is then based on the system
\eqref{eq_firstMaxwell}.


\section{Discontinuous Least Squares Method for Time-Harmonic
Equations}
\label{sec_method}
Aiming to construct a discontinuous least squares finite element method for
\eqref{eq_firstMaxwell}, we first define a least squares functional
based on \eqref{eq_firstMaxwell}, which reads:
\begin{equation}
  \begin{aligned}
    J_h(\bm{u}, & \bm{p}) :=  \sum_{K \in \MTh} \left( \| \nabla
    \times \bm{p} - k \bm{u} - \wt{\bm{f}} \|_{L^2(K)}^2 + \| \nabla
    \times \bm{u} - k \bm{p} \|_{L^2(K)}^2 \right) \\
    + & \sum_{f \in \MFh^i} \frac{\mu}{h_f} \left( \| \jump{\un \times
    \bm{u} }\|_{L^2(f)}^2 + \| \jump{\un \times \bm{p}} \|_{L^2(f)}^2
    \right) + \sum_{f \in \MFh^b} \frac{\mu}{h_f} \| \un \times \bm{u}
    - \un \times \bm{g} \|_{L^2(f)}^2,
  \end{aligned}
  \label{eq_functionalJ}
\end{equation}
where $\mu$ is a positive parameter and will be specified later on. We
note that in two dimensions the variable $\bm{p}$ in
\eqref{eq_functionalJ} is scalar-valued and in three dimensions the variable
$\bm{p}$ is a vector in $\mb{R}^3$. Then, we define two approximation
spaces $\bmr{V}_h^m$ and $\bmr{\Sigma}_h^m$ for the variables $\bm{u}$
and $\bm{p}$, respectively, 
\begin{displaymath}
  \begin{aligned}
    \bmr{V}_h^m & := (V_h^m)^d, \qquad \bmr{\Sigma}_h^m := (V_h^m)^{2d -
    3}, \\
  \end{aligned}
\end{displaymath}
where $V_h^m$ is the standard piecewise polynomial space, 
\begin{displaymath}
  \begin{aligned}
    V_h^m & := \left\{ v_h \in L^2(\Omega) \  | \  v_h|_K \in
    \mb{P}_m(K), \  \forall K \in \MTh \right\}. \\
  \end{aligned}
\end{displaymath}
Clearly, the functions in $\bmr{V}_h^m$ and $\bmr{\Sigma}_h^m$ can be
discontinuous across inter-element faces.  We seek the numerical
solution $(\bm{u}_h, \bm{p}_h) \in \bmr{V}_h^m \times
\bmr{\Sigma}_h^m$ by minimizing the functional \eqref{eq_functionalJ}
over the space $\bmr{V}_h^m \times \bmr{\Sigma}_h^m$, which reads:
\begin{equation}
  (\bm{u}_h, \bm{p}_h) = \mathop{\arg \min}_{ (\bm{v}_h, \bm{q}_h) \in
  \bmr{V}_h^m \times \bmr{\Sigma}_h^m} J_h(\bm{v}_h, \bm{q}_h).
  \label{eq_minJ}
\end{equation}
To solve the minimization problem \eqref{eq_minJ}, we can write the
corresponding Euler-Lagrange equation, which takes the form: {\it find
$(\bm{u}_h, \bm{p}_h) \in \bmr{V}_h^m \times \bmr{\Sigma}_h^m$ such
that}
\begin{equation}
  a_h(\bm{u}_h, \bm{p}_h; \bm{v}_h, \bm{q}_h) = l_h(\bm{v}_h,
  \bm{q}_h), \quad \forall (\bm{v}_h, \bm{q}_h) \in \bmr{V}_h^m \times
  \bmr{\Sigma}_h^m,
  \label{eq_bilinear}
\end{equation}
where the bilinear form $a_h(\cdot; \cdot)$ and the linear form
$l_h(\cdot)$ are defined as
\begin{equation}
  \begin{aligned}
    a_h(\bm{u}_h, \bm{p}_h; \bm{v}_h, \bm{q}_h) &: = \sum_{K \in \MTh}
    \int_K (\nabla \times \bm{p}_h - k\bm{u}_h) \cdot (\nabla \times
    \bm{q}_h - k\bm{v}_h) \d{x} \\
    &+ \sum_{K \in \MTh} \int_K (\nabla \times \bm{u}_h - k\bm{p}_h)
    \cdot (\nabla \times \bm{v}_h - k \bm{q}_h) \d{x}\\
    &+ \sum_{f \in \MFh^i} \frac{\mu}{h_f} \int_f \jump{ \un \times
    \bm{u}_h } \cdot \jump{ \un \times \bm{v}_h } \d{s} \\
    &+ \sum_{f \in \MFh^i} \frac{\mu}{h_f} \int_f \jump{ \un \times
    \bm{p}_h } \cdot \jump{ \un \times \bm{q}_h } \d{s} \\
    &+ \sum_{f \in \MFh^b} \frac{\mu}{h_f} \int_f (\un \times
    \bm{u}_h) \cdot (\un \times \bm{v}_h) \d{s},
  \end{aligned}
  \label{eq_bilinearform}
\end{equation}
and
\begin{displaymath}
  \begin{aligned}
  l_h(\bm{v}_h, \bm{q}_h) & := -\sum_{K \in \MTh} \int_K \bm{f} \cdot
  \bm{v}_h \d{x} + \sum_{K \in \MTh} \int_K \nabla \times \bm{q}_h
  \cdot \wt{\bm{f}} \d{x} \\
  &+ \sum_{f \in \MFh^b} \frac{\mu}{h_f} \int_f \un \times \bm{v}_h
  \cdot \bm{g} \d{s}.
  \end{aligned}
\end{displaymath}
Then we focus on the error estimate to the problem \eqref{eq_bilinear}. 
To do so, we first define the spaces $\bmr{V}_h$ and $\bmr{\Sigma}_h$
as
\begin{displaymath}
  \bmr{V}_h = \bmr{V}_h^m + H_0(\curl), \quad \bmr{\Sigma}_h =
  \bmr{\Sigma}_h^m + H(\curl).
\end{displaymath}
We introduce the following energy norms for both the spaces: 
\begin{displaymath}
  \begin{aligned}
    \unorm{\bm{u}}^2 &= \sum_{K \in \MTh} \left( \| \bm{u}
    \|_{L^2(K)}^2 + \| \nabla \times \bm{u} \|_{L^2(K)}^2 \right) +
    \sum_{f \in \MFh} \frac{1}{h_f} \| \jump{\un \times \bm{u} }
    \|_{L^2(f)}^2, \quad \forall \bm{u} \in \bmr{V}_h, 
  \end{aligned}
\end{displaymath}
and
\begin{displaymath}
  \begin{aligned}
    \pnorm{\bm{p}}^2 &= \sum_{K \in \MTh} \left( \| \bm{p}
    \|_{L^2(K)}^2 + \| \nabla \times \bm{p} \|_{L^2(K)}^2 \right) +
    \sum_{f \in \MFh^i} \frac{1}{h_f} \| \jump{\un \times \bm{p} }
    \|_{L^2(f)}^2, \quad \forall \bm{p} \in \bmr{\Sigma}_h, 
  \end{aligned}
\end{displaymath}
and we define the energy norm $\enorm{\cdot}$ as
\begin{displaymath}
  \enorm{(\bm{u}, \bm{p})}^2 = \unorm{\bm{u}}^2 + \pnorm{\bm{p}}^2,
  \quad \forall (\bm{u}, \bm{p}) \in \bmr{V}_h \times \bmr{\Sigma}_h.
\end{displaymath}
It is easy to see that $\unorm{\cdot}$ indeed defines a norm on
$\bmr{V}_h$ and $\pnorm{\cdot}$ indeed defines a norm on
$\bmr{\Sigma}_h$. As a result, $\enorm{\cdot}$ defines a norm on
$\bmr{V}_h \times \bmr{\Sigma}_h$.  Then we state the continuity
result of the bilinear form $a_h(\cdot ; \cdot)$ with respect to the
norm $\enorm{\cdot}$. 
\begin{lemma}
  Let the bilinear form $a_h(\cdot; \cdot)$ be defined as
  \eqref{eq_bilinearform} with any positive $\mu$, there exists a
  constant $C$ such that
  \begin{equation}
    |a_h(\bm{u}, \bm{p}; \bm{v}, \bm{q})| \leq C \enorm{ (\bm{u},
    \bm{p}) } \enorm{ (\bm{v}, \bm{q}) },
    \label{eq_continuity}
  \end{equation}
  for any $(\bm{u}, \bm{p}), (\bm{v}, \bm{q}) \in \bmr{V}_h \times
  \bmr{\Sigma}_h$.
  \label{le_continuity}
\end{lemma}
\begin{proof}
  Using the Cauchy-Schwartz inequality, we have that
  \begin{equation*}
    \begin{aligned}
      \sum_{K \in \MTh} \int_K k^2 \bm{u} \cdot \bm{v} d \bm{x}
      &\leq k^2 \left(\sum_{K \in \MTh} \| \bm{u} \|_{L^2(K)}^2
      \right)^{\frac{1}{2}} \left(\sum_{K \in \MTh} \| \bm{v}
      \|_{L^2(K)}^2 \right)^{ \frac{1}{2} } \\
      &\leq k^2 \enorm{(\bm{u} , \bm{p})} \enorm{(\bm{v}, \bm{q})}.
    \end{aligned}
  \end{equation*}
  Other terms that appear in the bilinear form \eqref{eq_bilinearform} 
  can be bounded similarly, which gives us the inequality
  \eqref{eq_continuity} and completes the proof.
\end{proof}
In order to prove the coercivity of the bilinear form $a_h(\cdot,
\cdot)$, we require the following lemmas. 
\begin{lemma}
  For any $\bm{u}_h \in \bmr{V}_h^m $, there exists a piecewise
  polynomial $\bm{v}_h \in \bmr{V}_h^m \cap H_0(\curl)$ such that
  \begin{equation}
    \begin{aligned}
      &\|\bm{u}_h - \bm{v}_h \|_{L^2(\Omega)}^2 \leq C \sum_{f \in
      \MFh} h_f \| \jump{\un \times \bm{u}_h} \|^2_{L^2(f)}, \\
      &\unorm{\bm{u}_h - \bm{v}_h }^2 \leq C \sum_{f \in \MFh}
      \frac{1}{h_f} \| \jump{\un \times \bm{u}_h} \|^2_{L^2(f)}, \\
    \end{aligned}
    \label{eq_projectionerror1}
  \end{equation}
  and for any $\bm{p}_h \in \bmr{\Sigma}_h^m$, there exists a
  piecewise polynomial $\bm{w}_h \in \bmr{\Sigma}_h^m \cap H(\curl)$
  such that
  \begin{equation}
    \begin{aligned}
    & \|\bm{p}_h - \bm{w}_h \|_{L^2(\Omega)}^2 \leq C \sum_{f \in
    \MFh^i} h_f \| \jump{\un \times \bm{p}_h} \|^2_{L^2(f)}, \\
    & \pnorm{\bm{p}_h - \bm{w}_h }^2 \leq C \sum_{f \in \MFh^i}
    \frac{1}{h_f} \| \jump{\un \times \bm{p}_h} \|^2_{L^2(f)}.  \\
    \end{aligned}
    \label{eq_projectionerror2}
  \end{equation}
  \label{le_projection}
\end{lemma}

\begin{proof}
  We prove this lemma by using the similar techniques as those in
  \cite{Houston2005interior, Karakashian2007convergence,
  li2019nondivergence}. In two dimensions, $\bm{p}_h$ is a
  scalar-valued piecewise polynomial function. For scalar-valued
  functions, the space $H(\curl)$ is equal to the space $H^1(\Omega)$.
  By \cite[Theorem 2.1]{Karakashian2007convergence}, there exists a
  piecewise polynomial $\bm{w}_h \in \bm{\Sigma}_h^m \cap H^1(\Omega)$
  satisfying the estimate \eqref{eq_projectionerror2}. For the
  vector-valued function $\bm{u}$, we will use \Ned's elements of
  the second type in two dimensions to prove the result and the
  estimate \eqref{eq_projectionerror1}. In three dimensions, $\bm{u}_h$
  and $\bm{p}_h$ are both vector-valued piecewise polynomials and we
  will apply 3D \Ned's elements of the second type to prove the two
  results. We primarily deduce that for the three-dimensional case and
  it is almost trivial to extend the proof to 2D for vector-valued
  functions.

  We first give some properties about \Ned's edge elements  of the
  second type, which was introduced by \Ned \cite{Nedelec1986new}. For
  a bounded domain $D$, we define the polynomial space $\mb{D}_{l}(D)$
  as $\mb{D}_l(D) = \mb{P}_{k -1}(D)^d \oplus \wt{\mb{P}}_{k - 1}(D)
  \bm{x}$, where $\wt{\mb{P}}_{k - 1}(D)$ is the space of homogeneous
  polynomials of degree $k - 1$.  For a tetrahedral element $K$ and a
  polynomial $\bm{v} \in \mb{P}(K)^d$, three types of moments (degrees
  of freedom) of the \Ned's elements associated with edges of $K$,
  faces of $K$ and $K$ itself are defined as follows,
  \begin{displaymath}
    \begin{aligned}
      M_K^e(\bm{t}) &= \left\{ \int_e (\bm{t} \cdot \bm{\tau}_e) q
      \d{s}, \quad \forall q \in \mb{P}_{m}(e) \right\}, \quad
      \text{for any edge } e \in \partial K,  \\
      M_K^f (\bm{t}) &= \left\{ \frac{1}{|f|} \int_f \bm{t} \cdot
      \bm{q} \d{s}, \quad \forall \bm{q} \in \mb{D}_{m - 1}(f)
      \right\}, \quad \text{for any face } f \in \partial K, \\
      M_K^b(\bm{t}) &= \left\{ \int_K \bm{t} \cdot \bm{q} \d{x}, \quad
      \forall \bm{q} \in \mb{D}_{m - 2}(K) \right\}, \\
    \end{aligned}
  \end{displaymath}
  where $\bm{\tau}_e $ denotes the unit vector along the edge. 
  Further, for any polynomial $\bm{t} \in \mb{P}_m(K)^d$, we define
  $t_{K, e}^i \in M_K^e(\bm{t})$, $t_{K, f}^i \in M_K^f(\bm{t})$ and
  $t_{K, b}^i \in M_K^b(\bm{t})$ as the corresponding moments of
  $\bm{t}$. We let $\left\{ \bm{\phi}_{K, e}^i \right\}$, $\left\{
  \bm{\phi}_{K, f}^i \right\}$, $\left\{ \bm{\phi}_{K, b}^i \right\}$
  are the Lagrange bases of the space $\mb{P}_m(K)^d$ with respect to the
  moments, respectively. Then, any polynomial $\bm{v} \in
  \mb{P}_m(K)^d$ can be uniquely expressed as 
  \[
    \bm{v} = \sum_{e \in \mc{E}(K)} \sum_{i = 1}^{N_e} v_{K, e}^i
    \bm{\phi}_{K, e}^i + \sum_{f \in \mc{F}(K)} \sum_{i = 1}^{N_f}
    v_{K, f}^i \bm{\phi}_{K, f}^i + \sum_{i = 1}^{N_b} v_{K, b}^i
    \bm{\phi}_{K, b}^i,
  \]
  where $\mc{E}(K)$ and $\mc{F}(K)$ are the sets of edges and faces of
  the element $K$, respectively. 
  
  Then we state the following estimates. For an element $K$ and any
  polynomial $\bm{v} \in \mb{P}_m(K)^d$, there exists a constant $C$ such
  that 
  \begin{equation}
    \begin{aligned}
       h_K^{-2} \| \bm{v} \|_{L^2(K)}^2 &+ \| \nabla \times \bm{v}
      \|_{L^2(K)}^2 \leq \\ C h_K^{-1} &\left( \sum_{e \in \mc{E}(K)}
      \sum_{i = 1}^{N_e} (v_{K, e}^i)^2 + \sum_{f \in \mc{F}(K)}
      \sum_{i = 1}^{N_f} (v_{K, f}^i)^2 + \sum_{i = 1}^{N_b} (v_{K,
        b}^i)^2 \right).
  \end{aligned}
    \label{eq_Nedinequality1}
  \end{equation}
  For an interior face $f \in \MFh^i$ shared by $K_1$ and $K_2$. For
  any polynomial $\bm{v}_1 \in \mb{P}_m(K_1)^d$ and $\bm{v}_2 \in
  \mb{P}_m(K_2)^d$, there exists a constant $C$ such that 
  \begin{equation}
    \sum_{i = 1}^{N_f} \left( {v}_{K_1, f}^i - {v}_{K_2, f}^i
    \right)^2 + \sum_{e \in \mc{E}(f)} \sum_{i = 1}^{N_e} \left(
      {v}_{K_1, e}^i - v_{K_2, e}^i \right)^2 \leq C
    \int_f | \un_f \times (\bm{v}_1 - \bm{v}_2) |^2 \d{\bm{s}}, 
    \label{eq_Nedinequality2}
  \end{equation}
  where $\mc{E}(f)$ is the set of edges belonging to the face $f$ and
  $\un_f$ is the unit outward normal on $f$. For a boundary face
  $f$, we let $K$ be the element such that $f \in \partial K$. Then,
  there exists a constant $C$ such that
  \begin{equation}
    \sum_{i = 1}^{N_f} \left( {v}_{K, f}^i\right)^2 + \sum_{e
      \in \mc{E}(f)} \sum_{i = 1}^{N_e} \left( {v}_{K, e}^i
    \right)^2 \leq C \int_f | \un_f \times \bm{v} |^2
    \d{\bm{s}}.
    \label{eq_Nedinequality3}
  \end{equation}
  The estimates  \eqref{eq_Nedinequality1}, \eqref{eq_Nedinequality2}
  and \eqref{eq_Nedinequality3} are obtained from the equivalence of
  norms over finite dimensional spaces and the scaling argument. 
  We refer to \cite{Houston2005interior, Monk2003finite} for the
  details of their proofs.
  
  Then we construct two new piecewise polynomials $\bm{w}_h \in
  \bmr{\Sigma}_h^m \cap H(\curl)$ and $\bm{v}_h \in \bmr{V}_h^m \cap
  H_0(\curl)$ satisfying the estimates \eqref{eq_projectionerror2} and
  \eqref{eq_projectionerror1}, respectively. To construct $\bm{w}_h$,
  for the element $K$, we define its moments $\left\{ w_{K, e}^i
  \right\} $, $ \left\{ w_{K, f}^i \right\} $ and $\left\{ w_{K, b}^i
  \right\} $ as follows,
  \begin{equation}
    w_{K, e}^i = \sum_{\wt{K} \in N(e)} \frac{1}{|N(e)|} p_{\wt{K},
    e}^i, \quad \forall e \in \MEh, \quad 1 \leq i \leq N_e,
    \label{eq_vKei}
  \end{equation}
  and
  \begin{equation}
    w_{K, f}^i = \sum_{\wt{K} \in N(f)} \frac{1}{|N(f)|} p_{\wt{K},
    f}^i, \quad \forall f \in \MFh, \quad 1 \leq i \leq N_f,
    \label{eq_vKfi}
  \end{equation}
  and 
  \begin{displaymath}
    w_{K, b}^i = p_{K, b}^i, \quad 1 \leq i \leq N_b.
  \end{displaymath}
  Here we let $N(e)$, $N(f)$ be the sets of elements that contain the
  edge $e$ and the face $f$, respectively, and we denote their
  cardinalities as $|N(e)|$ and $|N(f)|$. Clearly, the above moments
  will yield a piecewise polynomial $\bm{w}_h \in \bmr{\Sigma}_h^m
  \cap H(\curl)$. Then we focus on the error between $\bm{w}_h$ and
  $\bm{p}_h$. From \eqref{eq_Nedinequality1}, we deduce that 
  \begin{equation}
    \begin{aligned}
      & h_K^{\alpha - 1} \| \bm{p}_h - \bm{w}_h  \|_{L^2(K)}^2 + 
      h_K^{\alpha + 1} \| \nabla
      \times (\bm{p}_h - \bm{w}_h) \|_{L^2(K)}^2 \leq \\
      & C h_K^{\alpha} \left( \sum_{e \in \mc{E}(K)} \sum_{i =
      1}^{N_e} ({p}_{K, e}^i - {w}_{K, e}^i)^2 + \sum_{f \in
      \mc{F}(K)} \sum_{i = 1}^{N_f} (p_{K, f}^i - w_{K, f}^i)^2
      \right), \quad \alpha = \pm 1, \\
    \end{aligned}
    \label{eq_erroruhvhK}
  \end{equation}
  on the element $K$. By \eqref{eq_vKei}, we only need to consider the
  error for the edge $e$ which satisfies $|N(e)| \geq 2$. For such an
  edge $e$, we let $e$ be shared by a sequence of elements $ \left\{
  K_{e, 1}, K_{e, 2}, \ldots, K_{e, |N(e)|} \right\}$ with $K_{e, 1} =
  K$ and $K_{e, i}, K_{e, i + 1}$ are two neighbouring elements. Then
  the edge $e$ will be shared by $|N(e)| + 1$ faces, which are
  $\left\{ f_{e,1},f_{e,2},  f_{e, |N(e)| + 1} \right\}$ where $f_{e,
  i},f_{e,i+1}$ are the adjacent faces of the element $K_{e, i}$, and
  we further have that $f_{e, 2}, \ldots, f_{e, |N(e)|} \in \MFh^i$.
  Then from \eqref{eq_Nedinequality1} and the mesh regularity, we
  deduce that 
  \begin{equation}
    \begin{aligned}
      \sum_{i = 1}^{N_e} h_K^{\alpha} \left( p_{K, e}^i - w_{K, e}^i
      \right)^2 & \leq C \sum_{j = 2}^{|N(e)|} \sum_{i = 1}^{N_e}
      h_K^\alpha \left( p_{K_{e,1}, e}^i - w_{K_{e, j}, e}^i \right)^2
      \\ 
      & \leq C \sum_{j = 1}^{|N(e)| - 1} \sum_{i = 1}^{N_e}
      h_K^{\alpha} \left( p_{K_{e, j}, e}^i - w_{K_{e, j + 1}, e}^i
      \right)^2 \\
      & \leq C \sum_{j = 2}^{|N(e)|} \int_{f_{e, j}} h_{f_{e,
      j}}^{\alpha} \| \jump{ \un \times \bm{p}_h} \|^2 \d{s} \\ 
      &\leq C \sum_{f \in \mc{F}(e) \cap \MFh^i} h_f^{\alpha} \|
      \jump{\un \times \bm{p}_h} \|_{L^2(f)}^2,  \\
    \end{aligned}
    \label{eq_erroruiviK}
  \end{equation}
  where $\mc{F}(e)$ are the set of faces that contain the edge $e$.

  By \eqref{eq_vKfi}, we consider the errors on interior faces, and
  from \eqref{eq_Nedinequality2} we obtain that
  \begin{equation}
    \begin{aligned}
      \sum_{i = 1}^{N_f} h_K^\alpha \left( p_{K, f}^i - w_{K, f}^i
      \right)^2 & \leq C \sum_{K' \in N(f)} \sum_{i = 1}^{N_f}
      h_K^{\alpha} \left( p_{K, f}^i - p_{K', f}^i \right)^2 \\ 
      & \leq C h_f^{\alpha} \| \jump{\un \times \bm{p}_h}
      \|_{L^2(f)}^2, \\
    \end{aligned}
    \label{eq_errorufvfK}
  \end{equation}
  for the face $f \in \MFh^i$.  Combining the estimates
  \eqref{eq_erroruhvhK}, \eqref{eq_erroruiviK} and
  \eqref{eq_errorufvfK}, and summing over all elements arrives at the
  estimate \eqref{eq_projectionerror2}.

  The construction of the piecewise polynomial $\bm{v}_h$ is similar.
  We define its corresponding moments $\left\{ v_{K, e}^i \right\}$,
  $\left\{ v_{K, f}^i \right\}$ and $\left\{ v_{K, b}^i \right\}$ as 
  \[
    v_{K, e}^i = \begin{cases}
      \sum_{\wt{K} \in N(e)} \frac{1}{|N(e)|} u_{\wt{K}, e}^i, & \quad
      \forall e \in \MEh^i, \\
      0, & \quad \forall e \in \MEh^b, \\
    \end{cases}, \quad 1 \leq i \leq N_e,
    \label{eq_wKei}
  \]
  and
  \[
    v_{K, f}^i = \begin{cases}
      \sum_{\wt{K} \in N(f)} \frac{1}{|N(f)|}u_{\wt{K}, f}^i, \quad
      &\forall f \in \MFh^i, \\
      0, \quad & \forall f \in \MFh^b, \\
    \end{cases}, \quad 1 \leq i \leq N_f,
    \label{eq_wKfi}
  \]
  and
  \[
    v_{K, b}^i = u_{K, b}^i, \quad 1 \leq i \leq N_b.
  \]
  The above moments obviously imply that $\bm{v}_h \in H_0(\curl)$.
  We note that the moments of $\bm{v}_h$ and $\bm{w}_h$ are only
  different on boundary edges and faces. Hence the estimates
  \eqref{eq_errorufvfK} and \eqref{eq_erroruiviK} also hold for
  $\bm{v}_h$ on interior edges and faces. On the element $K$, we have
  that
  \[
    \sum_{i = 1}^{N_e} h_K^\alpha \left( u_{K, e}^i - v_{K, e}^i
    \right) \leq C \sum_{f \in \mc{F}(e)} h_f^{\alpha} \| \jump{\un
    \times \bm{u}_h} \|_{L^2(f)}^2,
  \]
  for an interior edge $e \in \mc{E}(K)$, and 
  \[
    \sum_{i = 1}^{N_f} h_K^\alpha \left( u_{K, f}^i - v_{K, f}^i
    \right)^2 \leq C h_f^{\alpha} \| \jump{\un \times \bm{u}_h}
    \|_{L^2(f)}^2,
  \]
  for an interior face $f \in \mc{F}(K)$. For boundary edges and
  faces, we directly apply the estimate \eqref{eq_Nedinequality3} to
  obtain the analogous inequalities, which bring us that
  \[
    \begin{aligned}
      h_K^{\alpha - 1} \| \bm{u}_h & - \bm{v}_h \|_{L^2(K)}^2  +
      h_K^{\alpha + 1} \| \nabla \times (\bm{u}_h - \bm{v}_h)
      \|_{L^2(K)}^2 \\
      & \leq C \left( \sum_{e \in \mc{E}(K)} \sum_{f \in
      \mc{F}(e)} h_f^{\alpha} \| \jump{\un \times \bm{u}_h}
      \|_{L^2(f)}^2 + \sum_{f \in \mc{F}(K)} h_f^\alpha \| \jump{\un
      \times \bm{u}_h } \|_{L^2(f)}^2 \right), \quad \alpha = \pm 1.
    \end{aligned}
  \]
  We finally arrive at the estimate \eqref{eq_projectionerror1} by
  summing over all elements. This completes the proof.
\end{proof}

\begin{lemma}
  There exists a constant $C$ such that
  \begin{equation}
    \unorm{ \bm{u}} + \pnorm{\bm{p}} \leq C \left( \| \nabla \times
    \bm{u} - k \bm{p} \|_{L^2(\Omega)} + \| \nabla \times \bm{p} -
    k \bm{u} \|_{L^2(\Omega)}  \right),
    \label{eq_Maxwellinequality}
  \end{equation}
  for any $\bm{u} \in \bmr{V}_h^m \cap H_0(\curl)$ and any $\bm{p} \in
  \bmr{\Sigma}_h^m \cap H(\curl)$.
  \label{le_Maxwellinequality}
\end{lemma}
\begin{proof}
  We let
  \begin{displaymath}
    \nabla \times \bm{u} - k\bm{p} = \bm{f}_1, \quad \nabla \times
    \bm{p} - k \bm{u} = \bm{f}_2.
  \end{displaymath}
  For any $\bm{\psi} \in H_0(\curl)$, we have 
  \begin{displaymath}
    \left( \nabla \times \bm{u}, \nabla \times \bm{\psi}
    \right)_{L^2(\Omega)} - k \left( \bm{p}, \nabla \times \bm{\psi}
    \right)_{L^2(\Omega)} = (\bm{f}_1, \nabla \times
    \bm{\psi})_{L^2(\Omega)},
  \end{displaymath}
  and
  \begin{displaymath}
    (\nabla \times \bm{p}, \bm{\psi})_{L^2(\Omega)} - k \left( \bm{u},
    \bm{\psi} \right) = \left( \bm{f}_2, \bm{\psi}
    \right)_{L^2(\Omega)}.
  \end{displaymath}
  Using the integration by parts to obtain that
  \begin{displaymath}
    \begin{aligned}
      (\nabla \times \bm{u}, \nabla \times \bm{\psi})_{L^2(\Omega)} -
      k^2 (\bm{u}, \bm{\psi})_{L^2(\Omega)} &= (\bm{f}_1, \nabla
      \times \bm{\psi})_{L^2(\Omega)} + k (\bm{f}_2,
      \bm{\psi})_{L^2(\Omega)}. \\
    \end{aligned}
  \end{displaymath}
  By \cite[Theorem 5.2]{Hiptmair2002finite}, we derive that
  \begin{displaymath}
    \begin{aligned}
      \| \bm{u} \|_{H(\curl)} &\leq C \sup_{\bm{\psi} \in H_0(\curl)}
      \frac{ (\nabla \times \bm{u}, \nabla  \times
      \bm{\psi})_{L^2(\Omega)} - k^2 (\bm{u},
      \bm{\psi})_{L^2(\Omega)}}{\| \bm{\psi} \|_{H(\curl)} } \\
      & \leq C \sup_{\bm{\psi} \in H_0(\curl)} \frac{ (\bm{f}_1,
      \nabla \times \bm{\psi})_{L^2(\Omega)} + k (\bm{f}_2,
      \bm{\psi})_{L^2(\Omega)}}{\| \bm{\psi} \|_{H(\curl)} } \\
      & \leq C \left(  \|\bm{f}_1 \|_{L^2(\Omega)} + \| \bm{f}_2
      \|_{L^2(\Omega)}  \right). \\
    \end{aligned}
  \end{displaymath}
  Since $\bm{u} \in H_0(\curl)$, we have that $ \| \bm{u}
  \|_{H(\curl)} = \unorm{\bm{u}}$. Further, we deduce that
  \begin{displaymath}
    \begin{aligned}
      \pnorm{\bm{p}} & = \| \bm{p} \|_{H(\curl)} \leq C \left( \|\bm{p}
      \|_{L^2(\Omega)} + \|\nabla \times \bm{p} \|_{L^2(\Omega)}
      \right) \\
      & \leq C \left( \| \bm{f}_1 \|_{L^2(\Omega)}
      + \| \nabla \times \bm{u} \|_{L^2(\Omega)} + \| \bm{f}_2
      \|_{L^2(\Omega)} + \|\bm{u} \|_{L^2(\Omega)} \right) \\
      & \leq C \left( \|\bm{f}_1 \|_{L^2(\Omega)} + \| \bm{f}_2
      \|_{L^2(\Omega)} \right), \\
    \end{aligned}
  \end{displaymath}
  which gives us the estimate \eqref{eq_Maxwellinequality} and
  completes the proof.
\end{proof}
Now we are ready to state that the bilinear form $a_h(\cdot; \cdot)$
is coercive with respect to the energy norms.
\begin{lemma}
  Let the bilinear form $a_h(\cdot; \cdot)$ be defined as
  \eqref{eq_bilinearform} with any positive $\mu$, there exists a
  constant $C$ such that  
  \begin{equation}
    a_h(\bm{u}_h, \bm{p}_h; \bm{u}_h, \bm{p}_h) \geq C \enorm{
    (\bm{u}_h, \bm{p}_h)}^2,
    \label{eq_coercivity}
  \end{equation}
  for any $(\bm{u}_h, \bm{p}_h) \in \bmr{V}_h^m \times \bmr{\Sigma}_h^m$.
  \label{le_coercivity}
\end{lemma}

\begin{proof}
  Clearly, we have that 
  \begin{displaymath}
    \begin{aligned}
      a_h(\bm{u}_h, \bm{p}_h; \bm{u}_h,& \bm{p}_h) =  \sum_{K \in
      \MTh} \left( \| \nabla \times \bm{p}_h - k \bm{u}_h
      \|_{L^2(K)}^2 + \| \nabla \times \bm{u}_h - k \bm{p}_h
      \|_{L^2(K)}^2 \right) \\
      + & \sum_{f \in \MFh^i} \frac{\mu}{h_f} \left( \| \jump{\un
      \times \bm{u}_h }\|_{L^2(f)}^2 + \| \jump{\un \times \bm{p}_h}
      \|_{L^2(f)}^2 \right) + \sum_{f \in \MFh^b} \frac{\mu}{h_f} \|
      \un \times \bm{u}_h \|_{L^2(f)}^2. \\
    \end{aligned}
  \end{displaymath}

  By Lemma \ref{le_projection}, for $\bm{u}_h$ and $\bm{p}_h$, there
  exists a polynomial $\bm{v}_h \in \bmr{V}_h^m \cap H_0(\curl)$ such
  that 
  \begin{displaymath}
    \begin{aligned}
      \unorm{\bm{u}_h - \bm{v}_h }^2 \leq C \sum_{f \in \MFh} 
      \frac{1}{h_f} \| \jump{\un \times \bm{u}_h} \|^2_{L^2(f)} 
      \leq C  a_h(\bm{u}_h, \bm{p}_h; \bm{u}_h, \bm{p}_h) , \\
    \end{aligned}
  \end{displaymath}
  and there exists a polynomial $\bm{q}_h \in \bmr{\Sigma}_h^m \cap
  H(\curl)$ such that 
  \begin{displaymath}
    \pnorm{\bm{p}_h - \bm{q}_h }^2 \leq C \sum_{f \in \MFh^i}
    \frac{1}{h_f} \| \jump{\un \times \bm{p}_h} \|_{L^2(f)}^2 \leq
    C a_h(\bm{u}_h, \bm{p}_h; \bm{u}_h, \bm{p}_h).  \\
  \end{displaymath}
  Hence, 
  \begin{displaymath}
    \begin{aligned}
      \enorm{ (\bm{u}_h, \bm{p}_h)}^2 & \leq C \left( \enorm{
      (\bm{u}_h - \bm{v}_h, \bm{p}_h - \bm{q}_h) }^2 + \enorm{
      (\bm{v}_h, \bm{q}_h)}^2 \right) \\
      & \leq C \left(  a_h(\bm{u}_h, \bm{p}_h; \bm{u}_h, \bm{p}_h) +
      \enorm{ (\bm{v}_h, \bm{q}_h)}^2 \right). 
    \end{aligned}
  \end{displaymath}
  By Lemma \ref{le_Maxwellinequality}, we get that
  \begin{displaymath}
    \begin{aligned}
      \enorm{ (\bm{v}_h, \bm{q}_h)}^2 & \leq \left( \unorm{\bm{v}_h} 
      + \pnorm{q_h} \right)^2 \\
      & \leq C \left(  \| \nabla \times \bm{v}_h - \bm{q}_h
      \|_{L^2(\Omega)} + \| \nabla \times \bm{q}_h - \bm{v}_h
      \|_{L^2(\Omega)} \right)^2. \\
    \end{aligned}
  \end{displaymath}
  We apply the triangle inequality to derive that
  \begin{displaymath}
    \begin{aligned}
      \|\nabla \times \bm{v}_h - \bm{q}_h \|_{L^2(\Omega)}^2 & \leq C
      \left( \| \nabla \times \bm{u}_h - \bm{p}_h \|_{L^2(\MTh)}^2 +
      \| \nabla \times (\bm{u}_h - \bm{v}_h) \|_{L^2(\MTh)}^2 +
      \|\bm{p}_h - \bm{q}_h \|_{L^2(\Omega)}^2 \right)\\
      & \leq  C \left( \| \nabla \times \bm{u}_h - \bm{p}_h
      \|_{L^2(\MTh)}^2 + \unorm{\bm{u}_h - \bm{v}_h}^2 + 
      \pnorm{\bm{p}_h - \bm{q}_h}^2 \right) \\
      & \leq C a_h(\bm{u}_h, \bm{p}_h; \bm{u}_h, \bm{p}_h).  \\
    \end{aligned}
  \end{displaymath}
  Similarly, we have that 
  \begin{displaymath}
    \| \nabla \times \bm{q}_h - \bm{v}_h \|_{L^2(\Omega)}^2 \leq C
    a_h(\bm{u}_h, \bm{p}_h; \bm{u}_h, \bm{p}_h).
  \end{displaymath}
  Combining all inequalities above, we arrive at 
  \begin{displaymath}
    a_h(\bm{u}_h, \bm{p}_h; \bm{u}_h, \bm{p}_h) \geq C \enorm{
    (\bm{u}_h, \bm{p}_h)}^2,
  \end{displaymath}
  which gives the estimate \eqref{eq_coercivity} and completes the
  proof.
\end{proof}

\begin{remark}
  We note that here we have attained the coercivity result of the
  bilinear form $a_h(\cdot; \cdot)$.  
  This property allows us to derive the error estimate from
  the Lax-Milgram framework instead of using the
  G$\mathring{\text{a}}$rding-type inequality. In addition, the
  estimate \eqref{eq_coercivity} holds true unconditionally without
  any constraint on the mesh size.
\end{remark}

Then we state the Galerkin orthogonality of the bilinear form
$a_h(\cdot; \cdot)$. 
\begin{lemma}
  \label{le_Galerkin}
  Let the bilinear form $a_h(\cdot; \cdot)$ be defined as
  \eqref{eq_bilinearform} with any positive $\mu$. Let $(\bm{u},
  \bm{p}) \in H(\curl) \times H(\curl)$ be the exact solution to
  \eqref{eq_firstMaxwell}, and let $(\bm{u}_h, \bm{p}_h) \in
  \bmr{V}_h^m \times \bmr{\Sigma}_h^m$ be the solution to
  \eqref{eq_bilinear}. Then, the following equation holds true
  \begin{equation}
    a_h(\bm{u} - \bm{u}_h, \bm{p} - \bm{p}_h; \bm{v}_h, \bm{q}_h) = 0,
    \label{eq_Galerkin}
  \end{equation}
  for any $(\bm{v}_h, \bm{q}_h) \in \bmr{V}_h^m \times \bmr{\Sigma}_h^m$.
\end{lemma}

\begin{proof}
  The regularity of the exact solution $(\bm{u}, \bm{p})$ directly
  brings us that
  \begin{displaymath}
    \jump{\un \times \bm{u}} = \bm{0}, \quad \jump{\un \times \bm{p}}
    = \bm{0}, \quad \text{on} \,\, \forall f \in \MFh^i.
  \end{displaymath}
  Hence,
  \begin{equation*}
    \begin{aligned}
      a_h(\bm{u} - \bm{u}_h,& \bm{p} - \bm{p}_h ; \bm{v}_h, \bm{q}_h)
      = \sum_{K \in \MTh} \int_K (\nabla \times (\bm{p} - \bm{p}_h) -
      k(\bm{u} - \bm{u}_h)) \cdot (\nabla \times \bm{q}_h - k
      \bm{v}_h) \d{x} \\
      & + \sum_{K \in \MTh} \int_K (\nabla \times (\bm{u} -
      \bm{u}_h) - k(\bm{p} - \bm{p}_h)) \cdot (\nabla \times \bm{v}_h
      - k\bm{q}_h) \d{x} \\
      &\,\, + \sum_{f \in \MFh^i} \frac{\mu}{h_f} \int_f \jump{\un
      \times \bm{u}_h} \cdot \jump{\un \times \bm{v}_h} d \bm{s} +
      \sum_{f \in \MFh^i} \frac{\mu}{h_f} \int_f \jump{\un \times
      \bm{p}_h} \cdot \jump{\un \times \bm{q}_h} \d{s} \\
      &\,\,+ \sum_{f \in \MFh^b} \frac{\mu}{h_f} \int_f (\un \times
      (\bm{u} - \bm{u}_h)) \cdot (\un \times \bm{v}_h) \d{s} \\
      &= \sum_{K \in \MTh} \int_K \wt{\bm{f}} \cdot (\nabla \times
      \bm{q}_h - k\bm{v}_h) \d{x} + \sum_{f \in \MFh^b}
      \frac{\mu}{h_f} \int_f \bm{g} \cdot (\un \times \bm{v}_h) \d{s}
      - a_h(\bm{u}_h, \bm{p}_h; \bm{v}_h, \bm{q}_h) \\
      &= l_h(\bm{v}_h, \bm{q}_h) - a_h(\bm{u}_h, \bm{p}_h; \bm{v}_h,
      \bm{q}_h) \\  
      &= 0,
    \end{aligned}
  \end{equation*}
  which yields the equation \eqref{eq_Galerkin} and 
  completes the proof.
\end{proof}

Finally, we state the {\it a priori} error estimate of the method
proposed in this section under the energy norm $\enorm{\cdot}$ .
\begin{theorem}
  Let $(\bm{u}, \bm{p})$ be the exact solution to 
  \eqref{eq_firstMaxwell}, which admits the following regularity,
  \begin{equation*}
    \begin{aligned}
      (\bm{u}, \bm{p}) &\in H^{m+1}(\Omega)^d \times
      H^{m+1}(\Omega)^{2d-3}, \\
      (\nabla \times \bm{u}, \nabla \times \bm{p}) &\in
      H^{m+1}(\Omega)^{2d-3} \times H^{m+1}(\Omega)^d,
    \end{aligned}
  \end{equation*}
  and let $(\bm{u}_h,\bm{p}_h) \in \bmr{V}_h^m \times
  \bmr{\Sigma}_h^m$ be the numerical solution to \eqref{eq_bilinear},
  and let the bilinear form $a_h(\cdot; \cdot)$ be defined as
  \eqref{eq_bilinearform} with any positive $\mu$. Then there exists a
  constant $C$ such that
  \begin{equation}
    \begin{aligned}
      &\enorm{(\bm{u} - \bm{u}_h, \bm{p} - \bm{p}_h)} \leq \\C h^m 
      \big( \|\bm{u} \|_{H^{m+1}(\Omega)} + &\| \nabla \times \bm{u}
        \|_{H^{m+1}(\Omega)} + \| \bm{p} \|_{H^{m+1}(\Omega)} + \|
        \nabla \times \bm{p} \|_{H^{m+1}(\Omega)} \big).
    \end{aligned}
    \label{eq_errorestimate}
  \end{equation}
  \label{th_errorestimate}
\end{theorem}
\begin{proof}
  By Lemma \ref{le_Galerkin}, we have that
  \begin{equation*}
    a_h(\bm{u} - \bm{u}_h, \bm{p} - \bm{p}_h; \bm{v}_h, \bm{q}_h) = 0,
    \quad \forall (\bm{v}_h, \bm{q_h}) \in \bmr{V}_h^m \times
    \bmr{\Sigma}_h^m,
  \end{equation*}
  Together with Lemma \ref{le_continuity} and Lemma
  \ref{le_coercivity}, we obtain that
  \begin{equation*}
  \begin{aligned}
    \enorm{(\bm{u}_h - \bm{v}_h, \bm{p}_h-\bm{q}_h)}^2 &\leq
    Ca_h(\bm{u}_h - \bm{v}_h, \bm{p}_h-\bm{q}_h; \bm{u}_h - \bm{v}_h,
    \bm{p}_h - \bm{q_h}) \\
    &= Ca_h(\bm{u} - \bm{v}_h, \bm{p}-\bm{q}_h; \bm{u}_h - \bm{v}_h,
    \bm{p}_h - \bm{q_h}) \\
    &\leq C \enorm{(\bm{u} - \bm{v}_h, \bm{p}-\bm{q}_h)}
    \enorm{(\bm{u}_h - \bm{v}_h, \bm{p}_h - \bm{q}_h)},
  \end{aligned}
  \end{equation*}
  for any $(\bm{v}_h, \bm{q}_h) \in \bmr{V}_h^m \times
  \bmr{\Sigma}_h^m$. We eliminate the term $\enorm{ (\bm{u} -
  \bm{v}_h, \bm{p}_h - \bm{q}_h)}$ on both sides and apply the
  triangle inequality to get that
  \begin{equation}
    \enorm{(\bm{u}-\bm{u}_h, \bm{p}-\bm{p}_h)} \leq C \inf_{(\bm{v}_h,
    \bm{q}_h) \in \bmr{V}_h^m \times \bmr{\Sigma}_h^m} \enorm{(\bm{u} -
    \bm{v}_h, \bm{p} - \bm{q}_h)}.
    \label{eq_Cea}
  \end{equation}
  We denote by $(\bm{u}_I = \Pi_N \bm{u}, \bm{p}_I = \Pi_N \bm{p}) \in
  \bmr{V}_h^m \times \bmr{\Sigma}_h^m$ the \Ned interpolants of the
  second kind to the exact solution $(\bm{u}, \bm{p})$.  Using the
  interpolation estimate \cite[Theorem 5.41]{Monk2003finite} and the
  trace estimate, we arrive at 
  \begin{equation*}
  \begin{aligned}
    \unorm{\bm{u} - \bm{u}_I} &\leq C h^m \left( \| \bm{u}
    \|_{H^{m+1}(\Omega)} + \| \nabla \times \bm{u}
    \|_{H^{m+1}(\Omega)} \right), \\
    \pnorm{\bm{p} - \bm{p}_I} &\leq C h^m \left( \| \bm{p}
    \|_{H^{m+1}(\Omega)} + \| \nabla \times \bm{p}
    \|_{H^{m+1}(\Omega)} \right).
  \end{aligned}
  \end{equation*}
  Substituting the above two estimates into \eqref{eq_Cea} implies the
  error estimate \eqref{eq_errorestimate}, which completes the proof.
\end{proof}

Notice that the framework of the least squares finite element method
has a natural mesh refinement indicator, which is exactly the least
squares functional defined as \eqref{eq_functionalJ}. Therefore, we
define an element indicator $\eta_K$ for the element $K$ as 
\begin{equation}
  \begin{aligned}
    \eta_K^2 :=   \|  \nabla \times \bm{p}_h - & k \bm{u}_h -
    \wt{\bm{f}} \|_{L^2(K)}^2  + \| \nabla \times \bm{u}_h - k
    \bm{p}_h \|_{L^2(K)}^2  \\
    + &  \sum_{f \in \MFh^i \cap \mc{F}(K)}  \frac{1}{h_f} \left( \|
    \jump{\un \times \bm{u}_h }\|_{L^2(f)}^2 + \| \jump{\un \times
    \bm{p}_h } \|_{L^2(f)}^2 \right) \\ 
    + & \sum_{f \in \MFh^b \cap \mc{F}(K)} \frac{1}{h_f} \| \un \times
    \bm{u}_h - \un \times \bm{g} \|_{L^2(f)}^2. \\
  \end{aligned}
  \label{eq_etaK}
\end{equation}
The adaptive procedure consists of loops of the standard form:
\begin{displaymath}
  \text{Solve} \  \rightarrow \  \text{Estimate}\  \rightarrow \
  \text{Mark} \ \rightarrow \  \text{Refine}.
\end{displaymath}
Ultimately, we present the following adaptive algorithm for solving
the time-harmonic Maxwell equations:
\begin{enumerate}[Step 1]
  \item Given the initial mesh $\mc{T}_0$ and a positive parameter
    $\theta$, and set the iteration number $l = 0$;
  \item Solve the time-harmonic Maxwell equations on the mesh
    $\mc{T}_l$;
  \item Obtain the error indicator $\eta_K$ for all $K \in \mc{T}_l$
    with respect to the numerical solutions from the Step 2;
  \item Find the minimal subset $\mc{M} \subset
    \mc{T}_l$ such that $\theta \sum_{K \in \mc{T}_l} \eta_K^2 \leq
    \sum_{K \in \mc{M}} \eta_K^2$ and mark all elements in $\mc{M}$.
  \item Refine all marked elements to generate the next level mesh
    $\mc{T}_{l + 1}$;
  \item If the stop criterion is not satisfied, then go to the Step 2 
    and set $l = l + 1$.
\end{enumerate}


\section{Numerical Results}
\label{sec_numericalresults}
Below we present several numerical examples in two dimensions and
three dimensions to demonstrate the performance of the proposed
method. In all these numerical experiments, the parameter $\mu$ in
\eqref{eq_functionalJ} is selected to be $1$ and we adopt the BiCGstab
solver together with the ILU preconditioner to solve the resulting
linear algebraic system.

\begin{figure}
  \centering
  \includegraphics[width=0.4\textwidth]{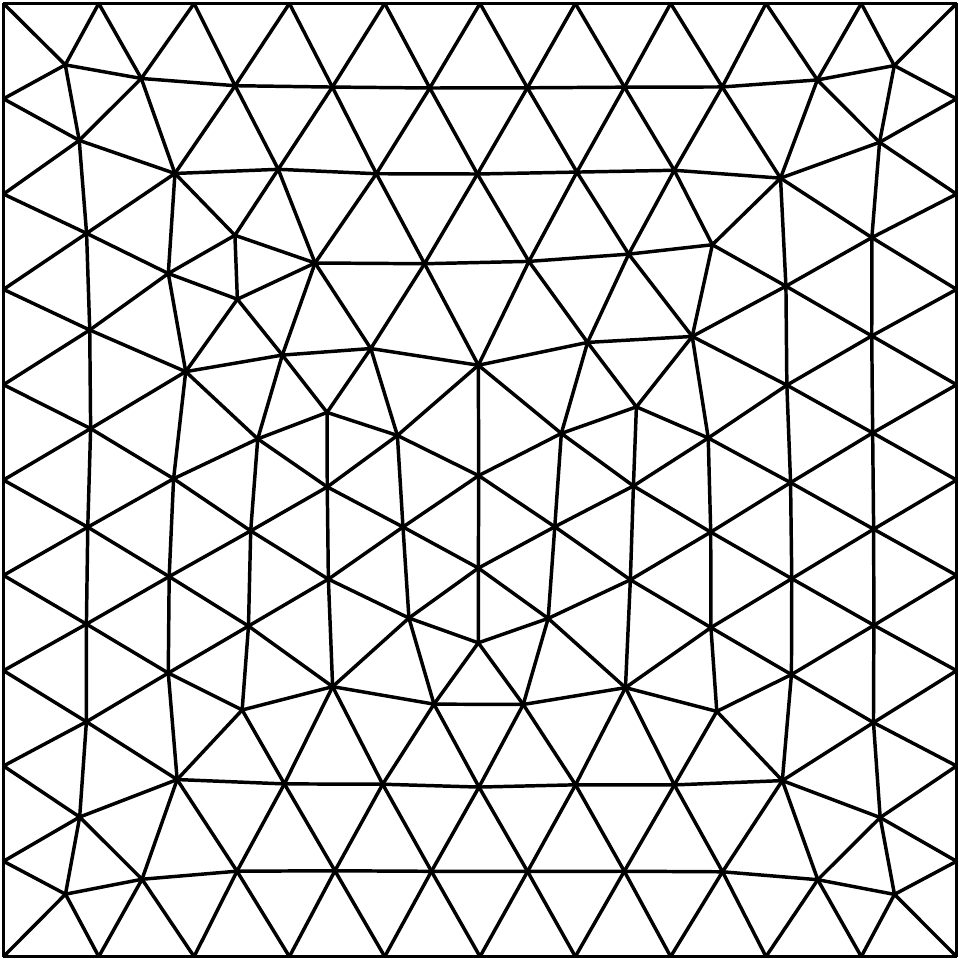}
  \hspace{25pt}
  \includegraphics[width=0.4\textwidth]{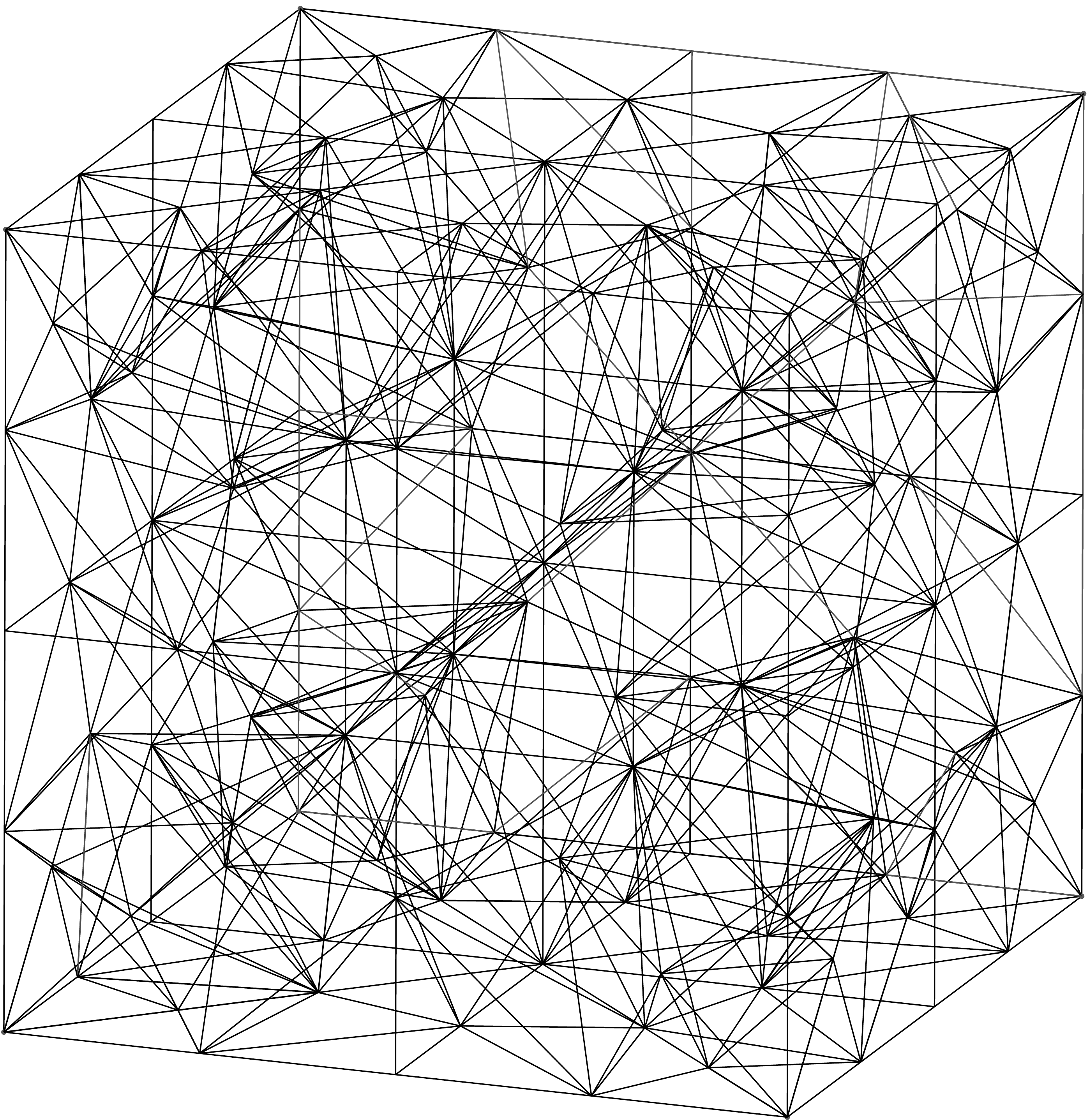}
  \caption{2d triangular partition with $ h = 1/10$ (left) / 3d
  tetrahedral partition with $h = 1/4$ (right).}
  \label{fig_partition}
\end{figure}

\noindent \textbf{Example 1.} For the first example, we consider a
smooth problem defined on the unit square domain $\Omega = (0, 1)^2$.
The exact solution to the time-harmonic Maxwell equations is given by
the smooth field \cite{Houston2005interior},
\begin{displaymath}
  \bm{u}(x, y) = \begin{bmatrix}
    \sin(ky) \\ \sin(kx)
  \end{bmatrix},
\end{displaymath}
and the source term $\bm{f}$ and the non-homogeneous boundary data
$\bm{g}$ are chosen accordingly. We solve this problem on a series of
shape-regular triangular meshes with the mesh size $h = 1/5$, $h =
1/10$, $\ldots$, $1/40$ (see Fig.~\ref{fig_partition}). The
convergence histories with the wave number $k = 1, 2, 8$ for the
accuracy $1 \leq m \leq 3$ are presented in Tab.~\ref{tab_ex1k1},
Tab.~\ref{tab_ex1k2} and Tab.~\ref{tab_ex1k8}, respectively. From
the numerical errors, we observe that the error under the energy norm
$\enorm{ (\bm{u} - \bm{u}_h, \bm{p} - \bm{p}_h)}$
converges to zero with the optimal speed $O(h^m)$ as the mesh size
approaches zero. In addition, for the $L^2$ errors, we can see that
$\|\bm{u} - \bm{u}_h \|_{L^2(\Omega)}$ and $ \| \bm{p} - \bm{p}_h
\|_{L^2(\Omega)}$ converge to zero at the rate $O(h^m)$ and
$O(h^{m+1})$, respectively, as the mesh is refined. As the wave number
$k$ increases, the errors in the approximation to the exact solution
also become larger under all error measurements. Here, we note that
the numerical results are consistent with those in
\cite{Houston2005interior}. In addition, all the computed convergence
rates agree with the error estimate given in Theorem
\ref{th_errorestimate}. 

\begin{table}
  \centering
    \renewcommand\arraystretch{1.3}
  \scalebox{1.}{
  \begin{tabular}{p{0.5cm} | p{3.3cm} | p{1.6cm} | p{1.6cm} | p{1.6cm}
    | p{1.6cm} | p{1cm}}
    \hline\hline 
    $m$ & mesh size & $1/10$ &  $1/20$ & $1/40$ & $1/80$ & order \\
    \hline
    \multirow{3}{*}{$1$} & $ \enorm{ (\bm{u} - \bm{u}_h, \bm{p} -
    \bm{p}_h)}$ 
    & 3.333e-2 & 1.667e-2 & 8.333e-3 & 4.166e-3 & 1.00 \\
    \cline{2-7} 
    & $ \|\bm{u} - \bm{u}_h\|_{L^2(\Omega)} $ 
    & 1.638e-2 & 8.168e-3 & 4.073e-3 & 1.663e-2 & 1.00 \\
    \cline{2-7} 
    & $ \|\bm{p} - \bm{p}_h\|_{L^2(\Omega)} $ 
    & 4.813e-4 & 1.208e-4 & 3.011e-5 & 7.512e-6 & 2.00 \\
    \hline
    \multirow{3}{*}{$2$} & $ \enorm{ (\bm{u} - \bm{u}_h, \bm{p} -
    \bm{p}_h)}$ 
    & 3.168e-4 & 7.922e-5 & 1.979e-5 & 4.950e-6 & 2.00 \\
    \cline{2-7} 
    & $ \|\bm{u} - \bm{u}_h\|_{L^2(\Omega)} $ 
    & 8.417e-5 & 2.102e-5 & 5.250e-6 & 1.311e-6 & 2.00 \\
    \cline{2-7} 
    & $ \|\bm{p} - \bm{p}_h\|_{L^2(\Omega)} $ 
    & 1.722e-6 & 2.156e-7 & 2.695e-8 & 3.369e-9 & 3.00 \\
    \hline
    \multirow{3}{*}{$3$} & $ \enorm{ (\bm{u} - \bm{u}_h, \bm{p} -
    \bm{p}_h)}$ 
    & 2.731e-6 & 3.402e-7 & 4.251e-8 & 5.312e-9 & 3.00 \\
    \cline{2-7} 
    & $ \|\bm{u} - \bm{u}_h\|_{L^2(\Omega)} $ 
    & 1.122e-6 & 1.383e-7 & 1.722e-8 & 2.155e-9 & 3.00 \\
    \cline{2-7} 
    & $ \|\bm{p} - \bm{p}_h\|_{L^2(\Omega)} $ 
    & 1.827e-8 & 1.140e-9 & 7.133e-11 & 4.472e-12 & 4.00 \\
    \hline
  \end{tabular}}
  \caption{Convergence history for Example 1 with $k=1$.}
  \label{tab_ex1k1}
\end{table}

\begin{table}
  \centering
    \renewcommand\arraystretch{1.3}
  \scalebox{1.}{
  \begin{tabular}{p{0.5cm} | p{3.3cm} | p{1.6cm} | p{1.6cm} | p{1.6cm}
    | p{1.6cm} | p{1cm}}
    \hline\hline 
    $m$ & $h$ & $1/10$ &  $1/20$ & $1/40$ & $1/80$ & order \\
    \hline
    \multirow{3}{*}{$1$} & $ \enorm{ (\bm{u} - \bm{u}_h, \bm{p} -
    \bm{p}_h)}$ 
    & 1.342e-1 & 5.920e-2 & 2.846e-2 & 1.408e-2 & 1.02 \\
    \cline{2-7} 
    & $ \|\bm{u} - \bm{u}_h\|_{L^2(\Omega)} $ 
    & 1.911e-2 & 4.873e-3 & 1.221e-3 & 3.049e-4 & 2.00 \\
    \cline{2-7} 
    & $ \|\bm{p} - \bm{p}_h\|_{L^2(\Omega)} $ 
    & 9.042e-2 & 3.033e-2 & 1.191e-2 & 5.446e-2 & 1.03 \\
    \hline
    \multirow{3}{*}{$2$} & $ \enorm{ (\bm{u} - \bm{u}_h, \bm{p} -
    \bm{p}_h)}$ 
    & 2.493e-3 & 6.229e-4 & 1.556e-5 & 3.891e-5 &  2.00 \\
    \cline{2-7} 
    & $ \|\bm{u} - \bm{u}_h\|_{L^2(\Omega)} $ 
    & 5.241e-4 & 1.307e-4 & 3.263e-5 & 8.149e-6 &  2.00 \\
    \cline{2-7} 
    & $ \|\bm{p} - \bm{p}_h\|_{L^2(\Omega)} $ 
    & 2.223e-5 & 2.473e-6 & 2.999e-7 & 3.722e-8 & 3.01 \\
    \hline
    \multirow{3}{*}{$3$} & $ \enorm{ (\bm{u} - \bm{u}_h, \bm{p} -
    \bm{p}_h)}$ 
    & 3.68e-5 & 4.610e-6 & 5.769e-7 & 7.264e-8 & 3.00 \\
    \cline{2-7} 
    & $ \|\bm{u} - \bm{u}_h\|_{L^2(\Omega)} $ 
    & 6.186e-6 & 7.759e-7 & 9.775e-8 & 1.228e-8 & 3.00 \\
    \cline{2-7} 
    & $ \|\bm{p} - \bm{p}_h\|_{L^2(\Omega)} $ 
    & 2.055e-7 & 1.289e-8 & 8.080e-10 &  5.061e-11 & 4.00 \\
    \hline
  \end{tabular}}
  \caption{Convergence history for Example 1 with $k = 2$.}
  \label{tab_ex1k2}
\end{table}

\begin{table}
  \centering
    \renewcommand\arraystretch{1.3}
  \scalebox{1.}{
  \begin{tabular}{p{0.5cm} | p{3.3cm} | p{1.6cm} | p{1.6cm} | p{1.6cm}
    | p{1.6cm} | p{1cm}}
    \hline\hline 
    $m$ & $h$ & $1/10$ &  $1/20$ & $1/40$ & $1/80$ & order \\
    \hline
    \multirow{3}{*}{$1$} & $ \enorm{ (\bm{u} - \bm{u}_h, \bm{p} -
    \bm{p}_h)}$ 
    & 7.531e-0 & 4.277e-0 & 1.652e-0 & 5.121e-1 & 1.69 \\
    \cline{2-7} 
    & $ \|\bm{u} - \bm{u}_h\|_{L^2(\Omega)} $ 
    & 6.482e-1 & 3.641e-1 & 1.382e-1 & 4.060e-2 & 1.76 \\
    \cline{2-7} 
    & $ \|\bm{p} - \bm{p}_h\|_{L^2(\Omega)} $ 
    & 6.432e-1 & 3.566e-1 & 1.337e-1 & 3.848e-2 & 1.80 \\
    \hline
    \multirow{3}{*}{$2$} & $ \enorm{ (\bm{u} - \bm{u}_h, \bm{p} -
    \bm{p}_h)}$ 
    & 2.787e-1 & 4.213e-2 & 9.898e-3 & 2.462e-3 & 2.00 \\
    \cline{2-7} 
    & $ \|\bm{u} - \bm{u}_h\|_{L^2(\Omega)} $ 
    & 2.079e-2 & 2.123e-3 & 4.336e-4 & 1.067e-4 & 2.02 \\
    \cline{2-7} 
    & $ \|\bm{p} - \bm{p}_h\|_{L^2(\Omega)} $ 
    & 1.914e-2 & 1.246e-3 & 7.962e-5 & 5.359e-6 & 3.80 \\
    \hline
    \multirow{3}{*}{$3$} & $ \enorm{ (\bm{u} - \bm{u}_h, \bm{p} -
    \bm{p}_h)}$ 
    & 9.493e-3 &  1.182e-3 &  1.478e-4 & 1.848e-5 & 3.00 \\
    \cline{2-7} 
    & $ \|\bm{u} - \bm{u}_h\|_{L^2(\Omega)} $ 
    & 3.841e-4 & 4.665e-5 &  5.875e-6 & 7.398e-7 & 3.00 \\
    \cline{2-7} 
    & $ \|\bm{p} - \bm{p}_h\|_{L^2(\Omega)} $ 
    & 9.612e-5 & 3.724e-6 & 2.219e-7 & 1.386e-8 & 4.00 \\
    \hline
  \end{tabular}}
  \caption{Convergence history for Example 1 with $k = 8$.}
  \label{tab_ex1k8}
\end{table}

\noindent \textbf{Example 2.} In this test, we consider a
three-dimensional problem in the unit cube $\Omega = (0, 1)^3$. We
solve the test problem on a series of tetrahedral meshes with the
resolution $h = 1/2$, $1/4$, $1/8$, $1/16$, see
Fig.~\ref{fig_partition}. The analytical solution is selected as  
\begin{displaymath}
  \bm{u}(x, y, z) = \begin{bmatrix}
    \sin(ky) \sin(kz) \\ \sin(kx) \sin(kz) \\ \sin(kx) \sin(ky)
  \end{bmatrix},
\end{displaymath}
and the data functions $\bm{f}$ and $\bm{g}$ are taken suitably. We
use the approximation spaces $\bmr{V}_h^m \times \bmr{\Sigma}_h^m$
with $1 \leq m \leq 3$ to approximate $\bm{u}$ and $\bm{p}$,
respectively. The numerical results with the wave number $k = 1$ are
shown in Tab.~\ref{tab_ex3d1}. We observe that under the energy norm
$\enorm{ \cdot}$ the numerical error tends to zero at the optimal
speed $O(h^m)$ as the mesh size decreases to zero. For the $L^2$
errors, our method shows a sub-optimal convergence rates for both
variables in three dimensions. We note that all numerical convergence
orders are still consistent with the theoretical error estimate.  

\begin{table}
  \centering
    \renewcommand\arraystretch{1.3}
  \scalebox{1.}{
  \begin{tabular}{p{0.5cm} | p{3.3cm} | p{1.6cm} | p{1.6cm} | p{1.6cm}
    | p{1.6cm} | p{1cm}}
    \hline\hline 
    $m$ & $h$ & $1/2$ &  $1/4$ & $1/8$ & $1/16$ & order \\
    \hline
    \multirow{3}{*}{$1$} & $ \enorm{ (\bm{u} - \bm{u}_h, \bm{p} -
    \bm{p}_h)}$ 
    & 2.410e-1 & 1.262e-1 & 6.433e-2 & 3.252e-1 & 1.00 \\
    \cline{2-7} 
    & $ \|\bm{u} - \bm{u}_h\|_{L^2(\Omega)} $ 
    & 9.731e-2 & 5.958e-2 & 3.239e-2 & 1.663e-2 & 0.99 \\
    \cline{2-7} 
    & $ \|\bm{p} - \bm{p}_h\|_{L^2(\Omega)} $ 
    & 3.343e-2 & 1.681e-2 & 8.051e-2 & 3.942e-2 & 1.00 \\
    \hline
    \multirow{3}{*}{$2$} & $ \enorm{ (\bm{u} - \bm{u}_h, \bm{p} -
    \bm{p}_h)}$ 
    & 3.638e-2 & 8.922e-3 & 2.221e-3 & 5.556e-4 & 2.00 \\
    \cline{2-7} 
    & $ \|\bm{u} - \bm{u}_h\|_{L^2(\Omega)} $ 
    & 1.795e-2 & 4.381e-3 & 1.098e-3 & 2.773e-4 & 1.99 \\
    \cline{2-7} 
    & $ \|\bm{p} - \bm{p}_h\|_{L^2(\Omega)} $ 
    & 7.575e-3 & 1.637e-3 & 3.767e-4 & 8.921e-5 & 2.05 \\
    \hline
    \multirow{3}{*}{$3$} & $ \enorm{ (\bm{u} - \bm{u}_h, \bm{p} -
    \bm{p}_h)}$ 
    & 1.876e-3 & 2.347e-4 & 2.953e-5 & 3.712e-6 & 3.00 \\
    \cline{2-7} 
    & $ \|\bm{u} - \bm{u}_h\|_{L^2(\Omega)} $ 
    & 6.916e-4 & 9.972e-5 & 1.339e-5 & 1.733e-6 & 2.96 \\
    \cline{2-7} 
    & $ \|\bm{p} - \bm{p}_h\|_{L^2(\Omega)} $ 
    & 6.650e-4 & 6.301e-5 & 6.513e-6 & 7.336e-7 & 3.15 \\
    \hline
  \end{tabular}}
  \caption{Convergence history for Example 2.}
  \label{tab_ex3d1}
\end{table}

\noindent \textbf{Example 3.} In this test, we investigate the
performance of the proposed method for dealing with the problem that
involves a singularity at the corner. The domain $\Omega$ is the
L-shaped domain $\Omega = (-1, 1)^2 \backslash [0, 1) \times (-1, 0]$
and we choose the exact solution, in polar coordinates $(r, \theta)$,
to be 
\begin{equation}
    \bm{u}(x, y) = \nabla ( (kr)^\alpha \sin(\alpha \theta)) +
    \begin{bmatrix}
    \sin(ky) \\ \sin(kx) \\
  \end{bmatrix},
  \label{eq_ex3singular}
\end{equation}
with $\alpha = 2/3$. Notice that the function $\bm{u}$ contains a
singularity at $(0, 0)$ and $\bm{u}$ only belongs to the space
$H^{\alpha - \varepsilon}(\Omega)$ for arbitrary small $\varepsilon$. 
The mesh size of the coarsest triangular mesh is $h = 1/5$ and we
uniformly refine the initial mesh for three times to solve the problem, 
see Fig.~\ref{fig_Ldomain}. We list the numerical errors under the
$L^2$ norms against the mesh size in Tab.~\ref{tab_exsingular}. From
the table, we can see that the error $ \| \bm{u} - \bm{u}_h
\|_{L^2(\Omega)}$ converges to zero at the rate $O(h^{\alpha})$ for all
$1 \leq m \leq 3$, which is in agreement with the regularity of the
function $\bm{u}$. For the variable $p$, the numerically detected
convergence rate is about $O(h^{\alpha + 2/3})$ in terms of the $L^2$
norm. The explanation of this convergence rate can be traced to the
L-shaped domain. We note that the observed convergence rates in this
example are consistent with the results in \cite{Houston2005interior,
Nguyen2011hybridizable}.

\begin{figure}
  \centering
  \includegraphics[width=0.4\textwidth]{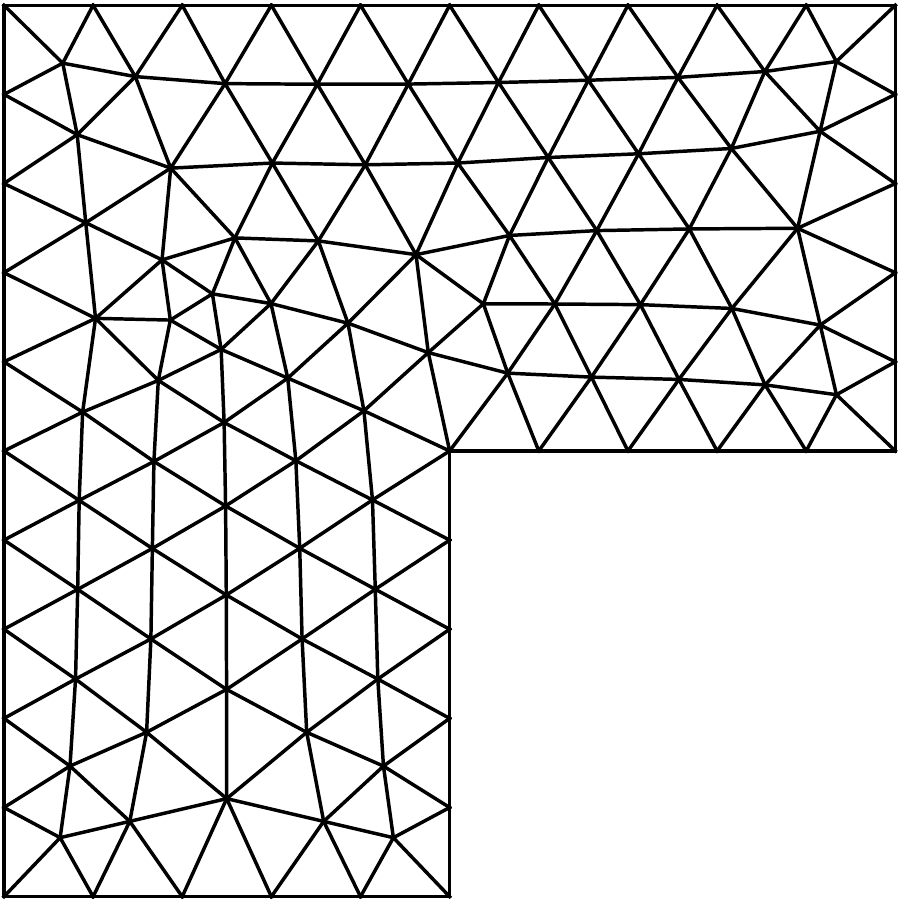}
  \hspace{25pt}
  \includegraphics[width=0.4\textwidth]{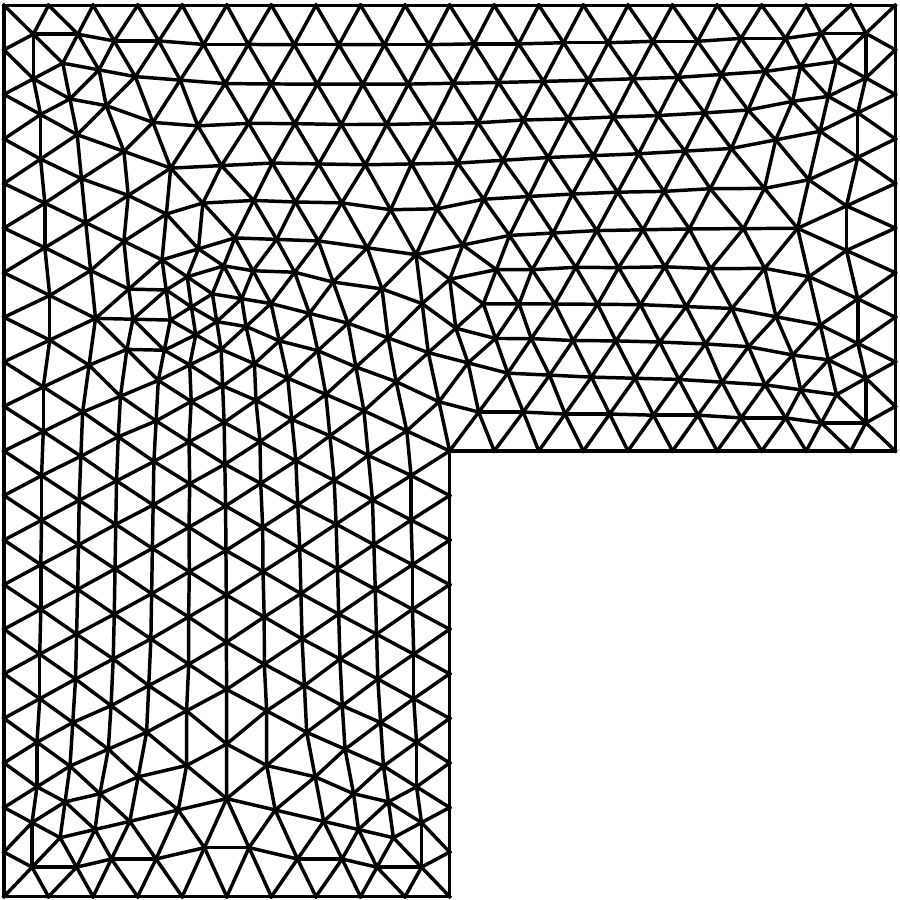}
  \caption{Triangular mesh for L-shaped domain with $h = 1/5$ (left) /
  $h = 1/10$ (right).}
  \label{fig_Ldomain}
\end{figure}

\begin{table}
  \centering
    \renewcommand\arraystretch{1.3}
  \scalebox{1.}{
  \begin{tabular}{p{0.5cm} | p{3.3cm} | p{1.6cm} | p{1.6cm} | p{1.6cm}
    | p{1.6cm} | p{1cm}}
    \hline\hline 
    $m$ & $h$ & $1/5$ &  $1/10$ & $1/20$ & $1/40$ & order \\
    \hline
    \multirow{2}{*}{$1$} 
    & $ \|\bm{u} - \bm{u}_h\|_{L^2(\Omega)} $ 
    & 8.817e-2 &  5.112e-2 &  3.002e-2 &  1.801e-2 & 0.73 \\
    \cline{2-7} 
    & $ \|\bm{p} - \bm{p}_h\|_{L^2(\Omega)} $ 
    & 2.115e-2 & 1.136e-2 & 5.039e-3 & 2.096e-3 & 1.26 \\
    \hline
    \multirow{2}{*}{$2$} 
    & $ \|\bm{u} - \bm{u}_h\|_{L^2(\Omega)} $ 
    & 4.001e-2 & 2.512e-2 &  1.581e-2 &  9.958e-3 & 0.67 \\
    \cline{2-7} 
    & $ \|\bm{p} - \bm{p}_h\|_{L^2(\Omega)} $ 
    & 5.598e-3 & 2.070e-3 &  8.020e-4 &  3.163e-4 & 1.34 \\
    \hline
    \multirow{2}{*}{$3$} & $ \|\bm{u} - \bm{u}_h\|_{L^2(\Omega)} $ 
    & 2.723e-2 & 1.718e-2 & 1.083e-2 & 6.820e-3 & 0.67 \\
    \cline{2-7} 
    & $ \|\bm{p} - \bm{p}_h\|_{L^2(\Omega)} $ 
    & 8.478e-4 & 3.331e-4 & 1.309e-4 & 5.121e-5 & 1.35 \\ 
    \hline
  \end{tabular}}
  \caption{Convergence history for Example 3.}
  \label{tab_exsingular}
\end{table}

\noindent \textbf{Example 4.} In this test, we solve a low-regularity
problem in three dimensions. We consider the unit cubic domain $\Omega
= (0, 1)^3$ and we select the analytical solution $\bm{u}$ as 
\begin{displaymath}
  \bm{u}(x, y, z) = \nabla (|\bm{x}|^\alpha) = \nabla ((x^2 + y^2 +
  z^2)^{\alpha / 2}),
\end{displaymath}
with $\alpha = 1.2$. The source function $\bm{f}$ and the
inhomogenuous boundary data $\bm{g}$ are selected accordingly, and the
wave number $k$ is set as $1$.  Clearly, $\bm{u}$ contains a
singularity near the corner $(0, 0, 0)$, which implies that the
function $\bm{u}$ lies in the space $H^{\alpha - 1/2 -
\varepsilon}(\Omega)$ with any $\varepsilon > 0$. We adopt the same
uniform meshes as in Example 2 to solve this problem. We list the
errors in approximation to $(\bm{u}, \bm{p})$ under the $L^2$ norms in
Tab.~\ref{tab_ex3dsingular}. The convergence rate of the $L^2$ error
$\| \bm{p} - \bm{p}_h \|_{L^2(\Omega)}$ is numerically detected to be
about $0.7$ for all $m$. The rate is perfectly in agreement with the
regularity of the function $\bm{u}$. For the $L^2$ error $\| \bm{u} -
\bm{u}_h \|_{L^2(\Omega)}$, the numerical convergence rate seems to be
$O(h)$ higher than the variable $\bm{p}$, which is optimal as the mesh
size tends to zero. We can not find this numerical phenomenon in the
smooth case (see Tab.~\ref{tab_ex3d1}). The optimal $L^2$ convergence
rate may depend on some properties of the exact solution and the mesh
size, and we hope the reason can be clarified in our future research. 

\begin{table}
  \centering
    \renewcommand\arraystretch{1.3}
  \scalebox{1.}{
  \begin{tabular}{p{0.5cm} | p{3.3cm} | p{1.6cm} | p{1.6cm} | p{1.6cm}
    | p{1.6cm} | p{1cm}}
    \hline\hline 
    $m$ & $h$ & $1/2$ &  $1/4$ & $1/8$ & $1/16$ & order \\
    \hline
    \multirow{2}{*}{$1$} 
    & $ \|\bm{u} - \bm{u}_h\|_{L^2(\Omega)} $ 
    & 4.082e-2 & 1.449e-2 & 4.836e-3 & 1.548e-3 & 1.65 \\
    \cline{2-7} 
    & $ \|\bm{p} - \bm{p}_h\|_{L^2(\Omega)} $ 
    & 5.446e-3 & 3.746e-3 & 2.375e-3 & 1.475e-3 & 0.68 \\
    \hline
    \multirow{2}{*}{$2$} 
    & $ \|\bm{u} - \bm{u}_h\|_{L^2(\Omega)} $ 
    & 1.399e-2 & 4.375e-3 & 1.351e-3 & 4.158e-4 & 1.69 \\
    \cline{2-7} 
    & $ \|\bm{p} - \bm{p}_h\|_{L^2(\Omega)} $ 
    & 2.996e-3 & 1.801e-3 & 1.112e-3 & 6.850e-4 & 0.70 \\
    \hline
    \multirow{2}{*}{$3$} & $ \|\bm{u} - \bm{u}_h\|_{L^2(\Omega)} $ 
    & 5.846e-3 & 1.803e-3 & 5.553e-4 & 1.709e-4 & 1.69 \\
    \cline{2-7} 
    & $ \|\bm{p} - \bm{p}_h\|_{L^2(\Omega)} $ 
    & 2.446e-3 & 1.478e-3 & 9.060e-4 & 5.578e-4 & 0.70 \\
    \hline
  \end{tabular}}
  \caption{Convergence history for Example 4.}
  \label{tab_ex3dsingular}
\end{table}

\textbf{Example 5.} In this example, we test our adaptive algorithm
proposed in Section \ref{sec_method}. We solve the low-regularity
problem in Example 3. The domain is the L-shaped domain (see
Fig.~\ref{fig_Lrefinement}) and the exact solution is chosen as
\eqref{eq_ex3singular}, which lies in the space $H^{2/3 -
\varepsilon}(\Omega)$ for any $\varepsilon > 0$. For the adaptive
algorithm, we choose the parameter $\theta = 0.25$ and we use the
longest-edge bisection algorithm to refine the mesh. We consider the
linear accuracy $\bmr{V}_h^1 \times \bmr{\Sigma}_h^1$ to solve the
problem. The convergence history under the $L^2$ norms for the
adaptive refinement is displayed in Fig.~\ref{fig_ex5error}. For this
problem, as we demonstrated in Tab.~\ref{tab_exsingular}, with the
uniform refinement the $L^2$ errors tends to zero at the speed
$O(N^{-0.67/2})$ and $O(N^{-1.35/2})$ for the variables $\bm{u}$ and
$\bm{p}$, respectively, where $N$ is the number of elements in the
partition. In Fig.~\ref{fig_ex5error}, the convergence orders of the
error $\|\bm{u} - \bm{u}_h \|_{L^2(\Omega)}$ and the error $\| \bm{p}
- \bm{p}_h \|_{L^2(\Omega)}$ approach zero seem to be about
$O(N^{-1/2})$ and $O(N^{-1})$, respectively. We note that these two
convergence rates match the convergence rates of the smooth case
in Example 1. Moreover, we depict the triangular mesh after 10
adaptive steps in Fig.~\ref{fig_Lrefinement}.  Obviously, the
refinement is pronounced around the corner $(0, 0)$ where the exact
solution contains a singularity.

\begin{figure}[htb]
  \centering
  \includegraphics[width=0.4\textwidth]{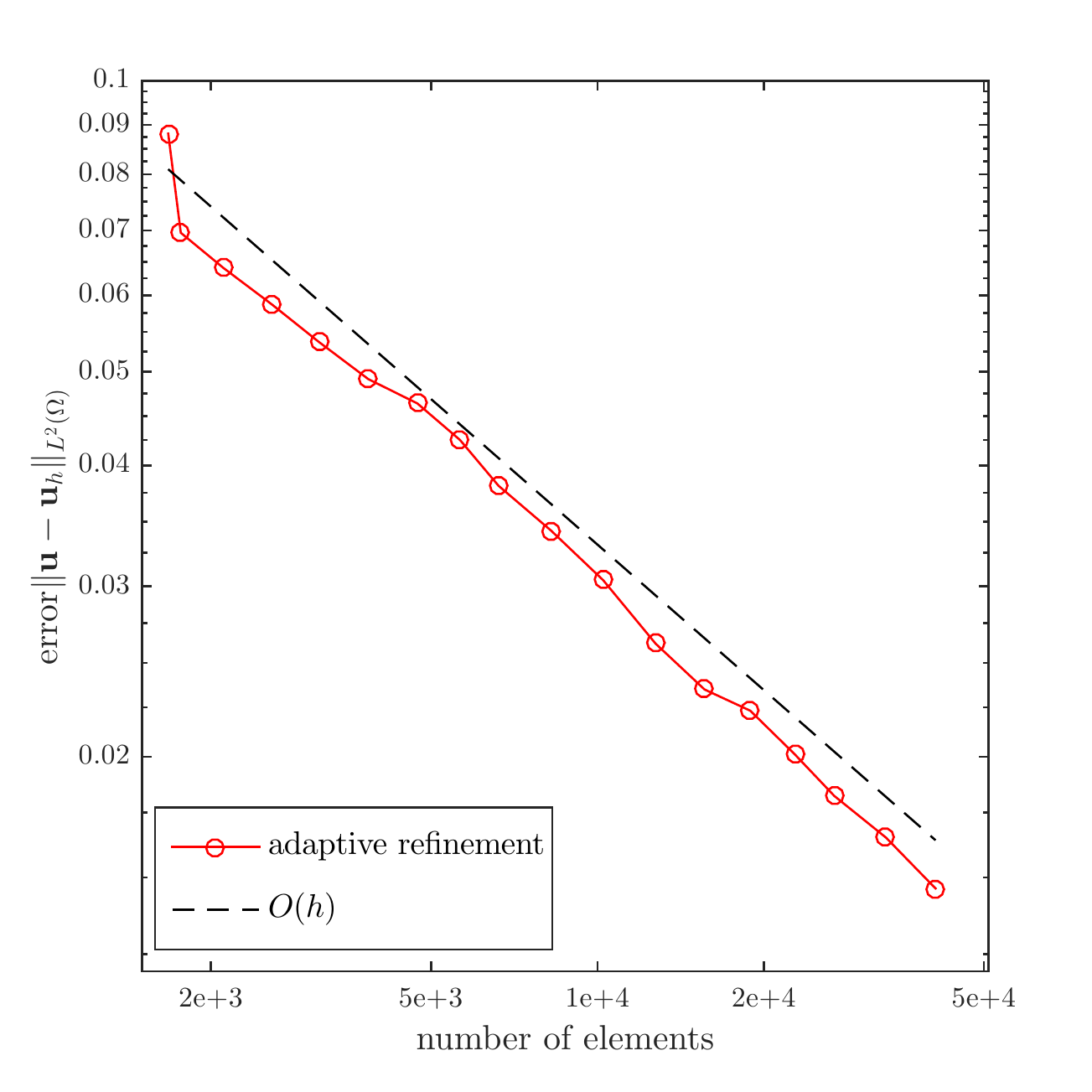}
  \includegraphics[width=0.4\textwidth]{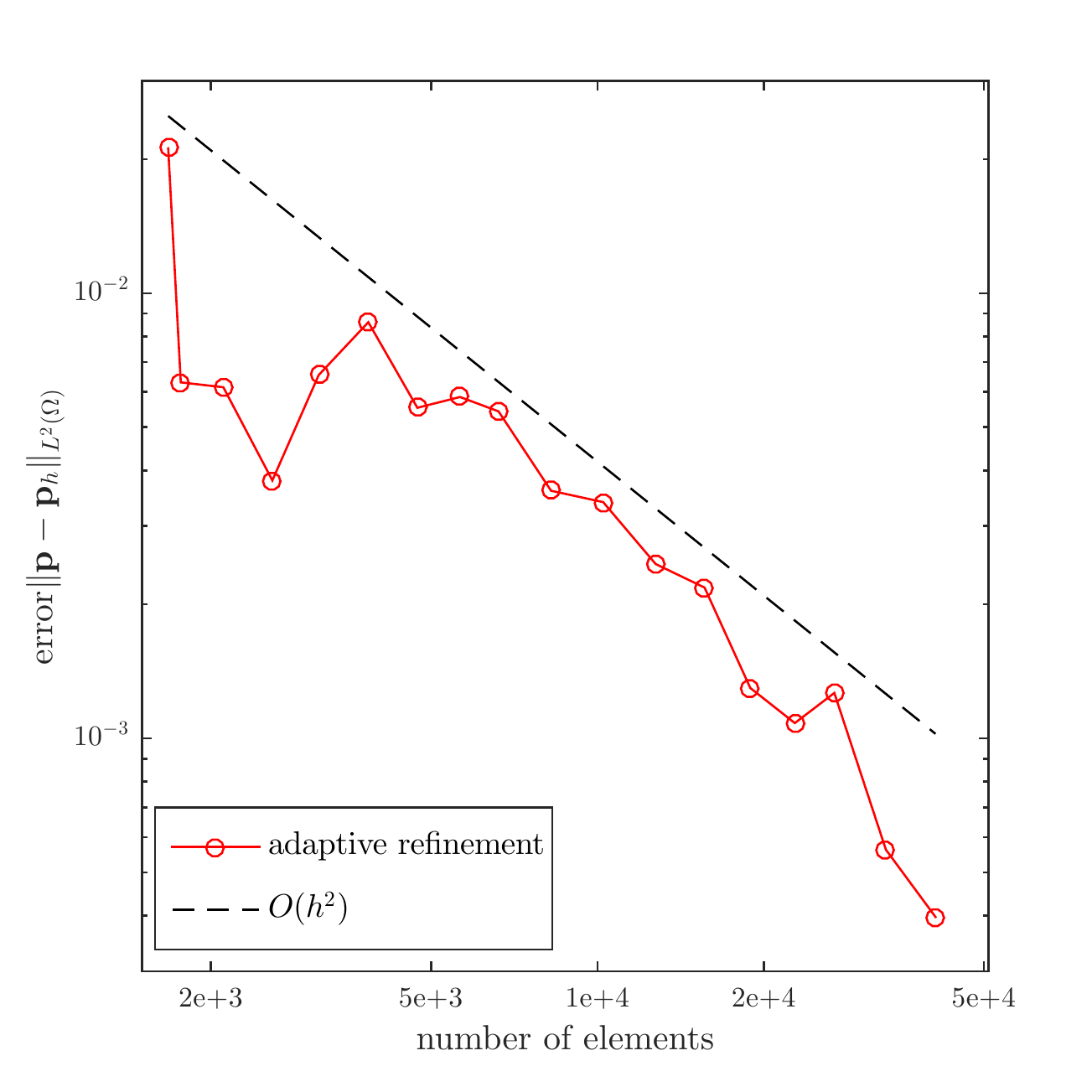}
  \caption{Convergence history for Example 5.}
  \label{fig_ex5error}
\end{figure}

\begin{figure}
  \centering
  \includegraphics[width=0.25\textwidth]{./figure/L0-crop.pdf}
  \hspace{15pt}
  \includegraphics[width=0.25\textwidth]{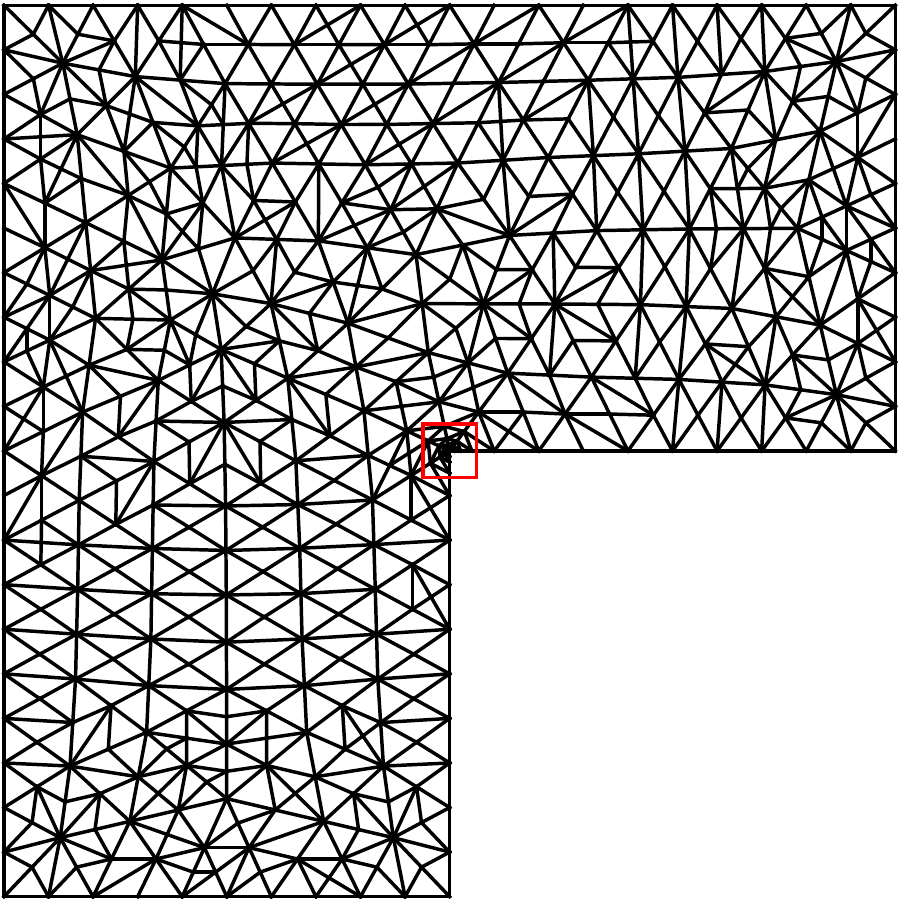}
  \hspace{15pt}
  \includegraphics[width=0.3\textwidth, height=0.25\textwidth]{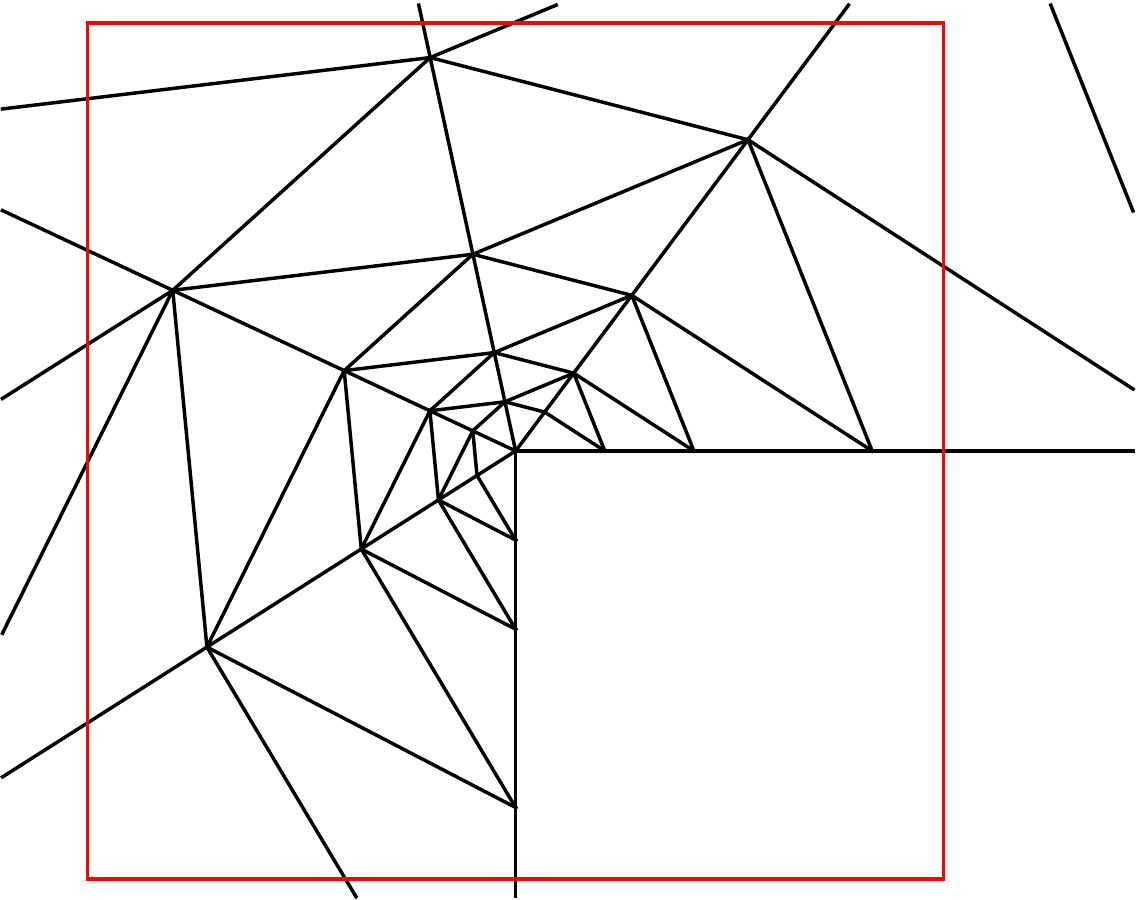}
  \caption{Initial mesh (left) / the triangular mesh after 10 adaptive
  refinement steps (middle) / the elements around the corner $(0, 0)$
  (right). }
  \label{fig_Lrefinement}
\end{figure}


\section{Conclusions}
We proposed a discontinuous least squares finite element method for
the time-harmonic Maxwell equations. Using discontinuous elements, we
designed a least squares functional with the weak imposition of the
tangential continuity on the interior faces. The convergence rates
were derived with respect to the energy norm and the $L^2$ norm.
Particularly, it was proved that our method is stable without any
constraint on the mesh size. Numerical results in two dimensions and
three dimensions illustrated the accuracy of our method.

\section*{Acknowledgements}
This research was supported by the Science Challenge Project (No.
TZ2016002) and the National Science Foundation in China (No.
11971041).


\bibliographystyle{amsplain}
\bibliography{../ref}

\end{document}